\documentclass[11pt,a4paper]{amsart}

\usepackage{amsmath,amscd}
\usepackage{amssymb}
\usepackage{amsthm}

\usepackage{graphicx}
\usepackage{enumitem}

\usepackage{mathrsfs}
\usepackage[ocgcolorlinks, linkcolor=blue]{hyperref}

\usepackage{calc}
             {\begin{list}{\arabic{enumi}.}{\usecounter{enumi}%
              \setlength{\labelsep}{0.5em}%
              \settowidth{\labelwidth}{(\arabic{enumi})}%
              \setlength{\leftmargin}{\labelwidth+\labelsep}}}%
             {\end{list}}

\usepackage{verbatim}            

\newtheorem{Theorem}{Theorem}[section]

\newtheorem{Lemma}[Theorem]{Lemma}

\theoremstyle{definition}
\newtheorem{Definition}[Theorem]{Definition}

\newtheorem*{RemarkNoNumber}{Remark}

\numberwithin{equation}{section}

\newcommand{\mR}{\mathbb{R}}                    
\newcommand{\mC}{\mathbb{C}}                    
\newcommand{\mZ}{\mathbb{Z}}                    
\newcommand{\mN}{\mathbb{N}}                    
\newcommand{\abs}[1]{\lvert #1 \rvert}          
\newcommand{\norm}[1]{\lVert #1 \rVert}         

\newcommand{\p}{\partial}

\newcommand{\eps}{\varepsilon}
\newcommand{\R}{\mathbb{R}}                    

\newcommand{\kommentar}[1]{}

\newcounter{sidenote}
\begin{document}

\title[Dimension bounds in monotonicity methods]{Dimension bounds in monotonicity methods for the Helmholtz equation} 

\author{Bastian Harrach}
\address{Institute for Mathematics, Goethe-University Frankfurt, Frankfurt am Main, Germany}
\email{harrach@math.uni-frankfurt.de}

\author{Valter Pohjola}
\address{Research Unit of Mathematical Sciences, University of Oulu, Oulu, Finland}
\email{valter.pohjola@gmail.com}

\author{Mikko Salo}
\address{University of Jyvaskyla, Department of Mathematics and Statistics, PO Box 35, 40014 University of Jyvaskyla, Finland}
\email{mikko.j.salo@jyu.fi}

\subjclass{35R30}


\keywords{inverse problems, Helmholtz equation, montonicity method}

\maketitle

\begin{abstract}
The article \cite{harrach2018helmholtz} established a monotonicity inequality for the Helmholtz
equation and presented applications to shape detection and local uniqueness in
inverse boundary problems. The monotonicity inequality states that if two
scattering coefficients satisfy $q_1 \leq q_2$, then the corresponding
Neumann-to-Dirichlet operators satisfy $\Lambda(q_1) \leq \Lambda(q_2)$ up to a
finite dimensional subspace. Here we improve the bounds for the
dimension of this space. In particular, if $q_1$ and $q_2$ have the same number
of positive Neumann eigenvalues, then the finite dimensional space is trivial.
\end{abstract}

\let\thefootnote\relax\footnotetext{\hrule \vspace{1ex} \centering This is a preprint version of a journal article published in\\
 \emph{SIAM J. Math. Anal.} \textbf{51}(4), 2995--3019, 2019
(\url{https://doi.org/10.1137/19M1240708}).
}

\section{Introduction} \label{sec_introduction}

This article is concerned with monotonicity properties arising in inverse problems and applications. As a basic example, if $\sigma_1$ and $\sigma_2$ are positive functions (representing electrical conductivities) in a bounded domain $\Omega \subset \mR^n$ and if $\Lambda(\sigma_1)$ and $\Lambda(\sigma_2)$ are the corresponding Neumann-to-Dirichlet operators (representing electrical boundary measurements), then one has the monotonicity property 
\[
\sigma_1 \leq \sigma_2 \ \implies \ \Lambda(\sigma_1) \geq \Lambda(\sigma_2).
\]
The last statement means that $\Lambda(\sigma_1) - \Lambda(\sigma_2)$ is a
positive semidefinite operator on the mean-free functions in $L^2(\p \Omega)$ (the so-called Loewner
order). This property, together with a certain nontrivial converse based on
localized potentials \cite{gebauer2008localized}, leads to efficient
monotonicity based methods for determining shapes of obstacles or inclusions
from electrical or optical boundary measurements, cf.~\cite{tamburrino2002new}
for the origin of this idea,
\cite{harrach2013monotonicity} for the proof of the converse monotonicity property, and the list of references for recent works on monotonicity-based methods at the end of this introduction. 

The recent work \cite{harrach2018helmholtz} extends monotonicity based methods to imaging problems with positive frequency, in particular acoustic imaging modelled by the Helmholtz equation. It turns out that the basic monotonicity property may fail in this case, but monotonicity still holds up to a finite dimensional space and \cite{harrach2018helmholtz} shows that shape detection methods and local uniqueness results can be developed also in this situation.
\cite{griesmaier2018monotonicity} extends this idea to farfield inverse scattering and shows numerical reconstructions.

Let us describe the results of \cite{harrach2018helmholtz} in more detail. Let
$\Omega \subset \mR^n$, $n \geq 2$ be a bounded Lipschitz domain, and let $q \in
L^{\infty}(\Omega)$ be a real valued function with $q \not\equiv 0$. Let $k >
0$, and consider the Neumann problem 
\begin{equation} \label{neumann_problem}
\left\{ \begin{array}{rcl} (\Delta + k^2 q) u \!\!\!&=&\!\!\! 0 \text{ in $\Omega$}, \\[5pt]
\partial_{\nu} u \!\!\!&=&\!\!\! g \text{ on $\partial \Omega$}. \end{array} \right.
\end{equation}
We assume that $k > 0$ is not a resonance frequency, which means that the Neumann problem has a unique solution $u \in H^1(\Omega)$ for any $g \in L^2(\p \Omega)$. Define the Neumann-to-Dirichlet (ND) operator 
\[
\Lambda(q): L^2(\partial \Omega) \to L^2(\partial \Omega), \ \ g \mapsto u|_{\partial \Omega}.
\]
Let $\lambda_1 \geq \lambda_2 \geq \cdots \to -\infty$ be the Neumann eigenvalues of $\Delta + k^2 q$ in $\Omega$, and let $d(q)$ be the number of positive Neumann eigenvalues (counted with multiplicity).

In \cite[Theorem 3.5]{harrach2018helmholtz} it was proved that 
\[
q_2 - q_1 \geq 0 \implies \Lambda(q_2)-\Lambda(q_1) \text{ has only finitely many negative eigenvalues}.
\]
Here we consider $\Lambda(q_2)-\Lambda(q_1)$ as a compact self-adjoint operator
on $L^2(\partial \Omega)$. Let $d(q_1, q_2)$ be the number of negative
eigenvalues of $\Lambda_{q_2} - \Lambda_{q_1}$ (counting with multiplicity). In
\cite[Theorem 3.5]{harrach2018helmholtz} it was also proved that $d(q_1, q_2)$
satisfies the bound 
\[
d(q_1, q_2) \leq d(q_2).
\]
The next result gives a more precise estimate for $d(q_1, q_2)$.

\begin{Theorem} \label{thm_monotonicity_dimension}
Let $q_1, q_2 \in L^{\infty}(\Omega) \setminus \{0\}$ be such that $k$ is not a
resonance frequency for $q_1$ and $q_2$. Assume that $q_1 \leq q_2$ a.e.\ in
$\Omega$. 
Then 
\[
d(q_1, q_2) \leq d(q_2) - d(q_1).
\]
\end{Theorem}

This has an immediate consequence: even if $q_1$ and $q_2$ are positive, the standard monotonicity inequality for the ND operators remains true if $q_1$ and $q_2$ have the same number of positive Neumann eigenvalues.

\begin{Theorem} \label{thm_monotonicity_dimension_consequence}
Let $q_1, q_2 \in L^{\infty}(\Omega) \setminus \{0\}$ be such that $k$ is not a
resonance frequency for $q_1$ and $q_2$. Assume that $d(q_1) = d(q_2)$. Then 
\[
q_1 \leq q_2 \implies \Lambda(q_1) \leq \Lambda(q_2).
\]
\end{Theorem}

Let us describe the main idea of the proof of Theorem \ref{thm_monotonicity_dimension}. If $q_1$ and $q_2$ are in $L^{\infty}(\Omega) \setminus \{0\}$ and satisfy $q_1 \leq q_2$, we define the interpolated potentials 
\[
q(t) = q_1 + t(q_2-q_1), \qquad t \in [0,1].
\]
Denote by $\lambda_1(t) \geq \lambda_2(t) \geq \cdots \to -\infty$ the Neumann
eigenvalues of $\Delta + k^2 q(t)$ in $\Omega$. Assume for simplicity that each
$\lambda_j(t)$ is a simple eigenvalue (the proof in Section \ref{sec_proof1}
removes this restriction). Then each map $\lambda_j: [0,1] \to \mR$ is smooth
and strictly increasing. This follows from the variational formula 
\[
\lambda_j'(t) = k^2 \int_{\Omega} (q_2-q_1) \varphi_j(t)^2 \,dx
\]
where $\{ \varphi_j(t) \}$ is an $L^2$-orthonormal basis consisting of Neumann eigenfunctions corresponding to $\{\lambda_j(t) \}$, and from the unique continuation principle. Now, when $t=0$ one starts with $d(q_1)$ positive eigenvalues, and when $t=1$ one has $d(q_2)$ positive eigenvalues. Since the maps $t \mapsto \lambda_j(t)$ are strictly increasing, exactly $d(q_2)-d(q_1)$ eigenvalues cross the real axis as $t$ increases to $1$, and the eigenspace at each crossing gives rise to a one-dimensional subspace of $L^2(\partial \Omega)$. Now if $g \in L^2(\partial \Omega)$ is orthogonal to all these one-dimensional subspaces, it follows that $((\Lambda(q_2)-\Lambda(q_1)) g, g) \geq 0$, proving that the finite-dimensional obstruction has dimension $\leq d(q_2)-d(q_1)$.

The next result complements Theorem \ref{thm_monotonicity_dimension} by showing that in certain cases where $q_1$ and $q_2$ differ by a constant, there are lower bounds on the number of negative eigenvalues. Its proof is based on computing an expression for the quadratic form $((\Lambda(q_2-b) -\Lambda(q_1+a)) g, g)$ in terms of the Neumann eigenfunctions of $\Delta + k^2 q_1$, and showing that the quadratic form is negative for $g$ in a space spanned by finitely many traces of Neumann eigenfunctions.

\begin{Theorem} \label{thm_monotonicity_dimension_optimal}
Let $q_1, q_2 \in L^{\infty}(\Omega) \setminus \{0\}$ be such that $k$ is non-resonant for $q_1$ and $q_2$. Assume that $q_2-q_1$ is a positive constant and $d(q_2) > d(q_1)$. Let $\mu_1$ be the largest negative Neumann eigenvalue of $\Delta + k^2 q_1$, and let $\mu_2$ be the smallest positive eigenvalue of $\Delta + k^2 q_2$. Let $N_j$ be the multiplicity of $\mu_j$.
\begin{enumerate}
\item[(a)] 
$\Lambda(q_2-b) - \Lambda(q_1+a)$ has at least $N_1$ negative eigenvalues whenever $a < k^{-2} |\mu_1|$ is sufficiently close to $k^{-2} |\mu_1|$, and $b \in [0,k^{-2} \mu_2)$.
\item[(b)] 
$\Lambda(q_2-b) - \Lambda(q_1+a)$ has at least $N_2$ negative eigenvalues whenever $b < k^{-2} \mu_2$ is sufficiently close to $k^{-2} \mu_2$, and $a \in [0,k^{-2} |\mu_1|)$.
\end{enumerate}
\end{Theorem}

From the previous theorem, we obtain the following special cases where equality is attained in Theorem \ref{thm_monotonicity_dimension}.

\begin{Theorem} \label{thm_monotonicity_dimension_example}
$\phantom{a}$
\begin{enumerate}
\item[(a)]
Let $q \in L^{\infty}(\Omega)$ be such that $0$ is a Neumann eigenvalue of $\Delta + k^2 q$ with multiplicity $N$. For $\eps > 0$ small enough, $d(q+\eps)-d(q-\eps) = N$ and $\Lambda(q+\eps) - \Lambda(q-\eps)$ has exactly $N$ negative eigenvalues.
\item[(b)]
Let $\Omega$ be the square $(0,\pi)^2 \subset \mR^2$, let $k > 0$, and let $N
\geq 2$ be even. There is a $c > 0$ such that for $\eps > 0$ small,
$\Lambda(c+\eps) - \Lambda(c-\eps)$ has exactly $N =
d(c+\eps)-d(c-\eps)$ negative eigenvalues.
\end{enumerate}
\end{Theorem}

Let us give some more references to earlier and related work, and comment on the relevance of our results. Monotonicity estimates and localized potentials techniques  have been used in different ways for the study of inverse problems \cite{harrach2009uniqueness,harrach2010exact,harrach2012simultaneous,arnold2013unique,harrach2013monotonicity,barth2017detecting,harrach2017local,brander2018monotonicity,griesmaier2018monotonicity,harrach2018helmholtz,harrach2018fractional,harrach2018localizing} and several recent works build practical reconstruction methods on monotonicity properties \cite{tamburrino2002new,harrach2015combining,harrach2015resolution,harrach2016enhancing,maffucci2016novel,tamburrino2016monotonicity,garde2017comparison,garde2017convergence,su2017monotonicity,ventre2017design,harrach2018monotonicity,zhou2018monotonicity,garde2019regularized}. Recently, monotonicity arguments were also discovered to yield Lipschitz stability results, cf.~\cite{harrach2019global,seo2018learning,harrach2019uniqueness}. All of these works consider stationary imaging cases where monotonicity of the ND operators holds in the sense of the Loewner order as explained above. So far, only \cite{harrach2018helmholtz,griesmaier2018monotonicity,harrach2019fractional} cover
the case of positive frequency imaging where the monotonicity only holds up to a finite dimensional space. For extending monotonicity-based theoretical uniqueness and stability results, as well as monotonicity-based numerical reconstruction methods, it seems to be of utmost importance to have a good bound on the number of eigenvalues that have to be disregarded. \cite{harrach2018helmholtz} showed that this number is smaller than $d(q_2)$, which might become arbitrarily large for high frequencies $k\to \infty$. Using this bound would result in disregarding a large part of the ND operators for high frequencies, and might make numerical reconstruction methods unfeasible. This article, however, shows that the number is smaller that $d(q_2)-d(q_1)$ which might still be small (or even zero) for high frequencies. Note also, that this article indicates that the bound is sharp for $q_1$ close to $q_2$, but that 
the bound might get too large when $q_2-q_1$ increases, cf.~section~\ref{sec_proof_constant}.

The rest of this paper is organized as follows. Section \ref{sec_proof1} gives the proof of Theorem \ref{thm_monotonicity_dimension}, and Section \ref{sec_proof2} proves Theorems \ref{thm_monotonicity_dimension_optimal} and \ref{thm_monotonicity_dimension_example}. Section \ref{sec_proof_constant} gives an simple alternative proof of Theorem \ref{thm_monotonicity_dimension} for the case where $q_1$ and $q_2$ are constant,
and numerically studies the sharpness of the bound for large $q_2-q_1$.

\subsection*{Acknowledgements}
M.S.\ was supported by the Academy of Finland (Centre of Excellence in Inverse Modelling and Imaging, grant numbers 312121 and 309963) and by the European Research Council under Horizon 2020 (ERC CoG 770924).


\section{Upper bound for the number of negative eigenvalues} \label{sec_proof1}

For the proof of Theorem \ref{thm_monotonicity_dimension}, it will be useful to consider solutions of the Helmholtz equation also when $k$ is a resonant frequency. In this case the Neumann data needs to satisfy finitely many linear constraints.

\begin{Lemma} \label{lemma_ndmap_general}
Let $q \in L^{\infty}(\Omega)$ and $k > 0$, and define the sets 
\begin{align*}
N(q) &= \{ \varphi \in H^1(\Omega) \,;\, (\Delta + k^2 q) \varphi = 0 \text{ in
$\Omega$}, \ \partial_{\nu} \varphi|_{\p\Omega} = 0 \}, \\
D(q) &= \{ \varphi|_{\partial \Omega} \,;\, \varphi \in N(q) \}
\end{align*}
Also let $N(q)^{\perp}$ be the orthogonal complement of $N(q)$ in $L^2(\Omega)$, and
let $D(q)^{\perp}$ be the orthogonal complement of $D(q)$ in $L^2(\p \Omega)$.
\begin{enumerate}
\item[(a)] 
$N(q)$ and $D(q)$ are finite-dimensional spaces whose dimension is the
multiplicity of $0$ as the Neumann eigenvalue of $\Delta + k^2 q$ in $\Omega$.
\item[(b)] 
For any $F \in L^2(\Omega)$ and $g \in L^2(\p \Omega)$, the equation 
\begin{equation} \label{helmholtz_f_g}
(\Delta + k^2 q) u = F \text{ in $\Omega$}, \qquad \p_{\nu} u|_{\p \Omega} = g,
\end{equation}
has a solution $u \in H^1(\Omega)$ if and only if one has the compatibility conditions 
\begin{equation} \label{helmholtz_compatibility_conditions}
\int_{\Omega} F \varphi \,dx = \int_{\p \Omega} g \varphi \,dS, \qquad \varphi \in N(q).
\end{equation}
In particular, a solution exists whenever $F \in N(q)^{\perp}$ and $g \in
D(q)^{\perp}$. The solution is unique up to addition of a function in $N(q)$,
and one has a bounded map 
\[
T_q: N(q)^{\perp} \times D(q)^{\perp} \to H^1(\Omega), \ \ (F,g) \mapsto u_{F,g}
\]
where $u_{F,g}$ is the unique solution with $(u, \varphi)_{H^1(\Omega)} = 0$ for $\varphi \in N(q)$.
\item[(c)] 
Let $\lambda_1 \geq \lambda_2 \geq \lambda_3 \geq \ldots \to -\infty$ be the
Neumann eigenvalues of $\Delta + k^2 q$ in $\Omega$ and let
$(\varphi_j)_{j=1}^{\infty}$ be a corresponding orthonormal basis of
$L^2(\Omega)$ consisting of Neumann eigenfunctions. If $F \in L^2(\Omega)$ and
$g \in L^2(\p \Omega)$ satisfy \eqref{helmholtz_compatibility_conditions}, then
any solution $u \in H^1(\Omega)$ of \eqref{helmholtz_f_g} may be represented as
the $L^2(\Omega)$-convergent sum 
\[
u = \sum_{j \in J} a_j \varphi_j + \sum_{j \notin J} \frac{1}{\lambda_j} \left[ \int_{\Omega} F \varphi_j \,dx - \int_{\p \Omega} g \varphi_j \,dS \right] \varphi_j,
\]
where $J = \{j \geq 1 \,;\, \lambda_j = 0 \}$ is finite, and $a_j \in \mR$ are some constants.
\end{enumerate}
\end{Lemma}

\begin{RemarkNoNumber}
The sum in part (c) may not converge in higher norms in general. In fact, if it did converge in some space where the normal derivative operator is bounded, then one would get that $\p_{\nu} u|_{\p \Omega} = 0$, which is not true if $g \neq 0$.
\end{RemarkNoNumber}

\begin{proof}
As in \cite[Section 2.1]{harrach2018helmholtz} we use the compact inclusion map $\iota: H^1(\Omega) \to
L^2(\Omega)$, to define $K = \iota^* \iota$ and $K_q =
\iota^* M_q \iota$, where $M_q: L^2(\Omega) \to L^2(\Omega)$ is the multiplication operator by $q$.
Both $K$ and $K_q$ are compact self-adjoint operators from $H^1(\Omega)$ to
$H^1(\Omega)$. A function $u \in H^1(\Omega)$ is a weak solution of
\eqref{helmholtz_f_g} if and only if 
\[
(\mathrm{Id} - K - k^2 K_q)u = -\iota^* F + \gamma^* g
\]
where $\gamma: H^1(\Omega) \to L^2(\partial \Omega)$ is the trace operator. By Fredholm theory (see e.g.\ \cite[Corollary 8.95]{renardy2004introduction}), this problem has a solution $u \in H^1(\Omega)$ for given $F \in L^2(\Omega)$, $g \in L^2(\partial \Omega)$ if and only if 
\[
(-\iota^* F + \gamma^* g, \varphi)_{H^1(\Omega)} = 0
\]
for all $\varphi \in H^1(\Omega)$ in the kernel of $\mathrm{Id} - K - k^2 K_q$. But this finite dimensional kernel is equal to $N(q)$, showing that \eqref{helmholtz_f_g} is solvable if and only if \eqref{helmholtz_compatibility_conditions} holds. The representation in \cite[Corollary 8.95]{renardy2004introduction} shows that there is a unique solution $u_{F,g}$ with $(u_{F,g}, \varphi)_{H^1(\Omega)} = 0$ for $\varphi \in N(q)$, and that 
\[
\norm{u_{F,g}}_{H^1(\Omega)} \leq C \norm{-\iota^* F + \gamma^* g}_{H^1(\Omega)} \leq C(\norm{F}_{L^2(\Omega)} + \norm{g}_{L^2(\p \Omega)}).
\]
Finally, the map $N(q) \to D(q)$, $\varphi \mapsto \varphi|_{\partial \Omega}$ is bijective by the unique continuation principle. This proves (a) and (b).

To prove (c) let $u$ be a solution of \eqref{helmholtz_f_g}. Since $(\varphi_j)$ is an orthonormal basis of $L^2(\Omega)$ we have 
\[
u = \sum_{j=1}^{\infty} c_j \varphi_j, \qquad c_j =  \int_{\Omega} u \varphi_j \,dx,
\]
with convergence in $L^2(\Omega)$. Testing the weak form of \eqref{helmholtz_f_g} against $\varphi_j$ and integrating by parts gives that 
\begin{align}
\int_{\Omega} F \varphi_j \,dx &= ( (\Delta+k^2 q)u, \varphi_j )_{L^2(\Omega)} = \int_{\p \Omega} g \varphi_j \,dS + (u, (\Delta+k^2 q) \varphi_j)_{L^2(\Omega)} \notag \\
 &=  \int_{\p \Omega} g \varphi_j \,dS + \lambda_j \int_{\Omega} u \varphi_j. \notag 
\end{align}
This yields the representation for $u$ in (c).
\end{proof}

For $q_1, q_2 \in L^{\infty}(\Omega) \setminus \{0\}$, we define 
\[
q(t) : = q_1 + t(q_2-q_1), \qquad t \in [0,1].
\]
We also define the family of  operators
\[
H(t) := \Delta + k^2 q(t),  \qquad t \in [0,1].
\]
The following result from analytic perturbation theory is needed to describe the behaviour of the eigenvalues of $H(t)$ as $t$ changes.

\begin{Lemma} \label{ev_pert}
Let $q_1, q_2 \in L^{\infty}(\Omega)$ and $k > 0$, assume that $q_1 \leq q_2$ and $q_1 \not\equiv q_2$, and let $\lambda_1 \geq \lambda_2 \geq \ldots \to -\infty$ be the Neumann eigenvalues of $\Delta + k^2 q_1$ in $\Omega$. There exist real-analytic functions $\lambda_j: [0,1] \to \mR$ and $\varphi_j: [0,1] \to L^2(\Omega)$ with the following properties:
\begin{enumerate}
\item[(a)] 
$\lambda_j(0) = \lambda_j$ for $j \geq 1$, and for any $t \in [0,1]$, the
numbers $\lambda_j(t)$ represent the repeated\footnote{We call 
$\lambda_1 \geq \lambda_2 \geq \ldots$ the repeated eigenvalues iff
the value $\lambda_j$ is repeated the times of its multiplicity in the sequence.}
Neumann eigenvalues of $H(t)$ in
$\Omega$. Zero is a Neumann eigenvalue of $H(t_0)$ with multiplicity $N$ if and
only if precisely $N$ functions $\lambda_j(t)$ vanish at $t=t_0$.
\item[(b)]
Each $\lambda_j(t)$ is strictly increasing on $[0,1]$.
\item[(c)] 
Each $\varphi_j(t)$ is a Neumann eigenfunction in $H^1(\Omega)$ satisfying 
\[
H(t) \varphi_j(t) = \lambda_j(t) \varphi_j(t) \text{ in $\Omega$},\qquad \p_{\nu} \varphi_j(t)|_{\p \Omega} = 0,
\]
and $(\varphi_j(t))_{j=1}^{\infty}$ is an orthonormal basis of $L^2(\Omega)$
for any $t \in [0,1]$. Each map $[0,1] \to H^1(\Omega), \ t \mapsto
\varphi_j(t)$ is real-analytic.
\end{enumerate}
\end{Lemma}
\begin{proof}
This result will be proved by analytic perturbation theory, and hence in this
proof we will assume the function spaces to be complex valued.

(a) For $z \in \mC$, define the operator 
\[
H(z) = \Delta + k^2 (q_1 + z(q_2-q_1)).
\]
Then $H(t) = \Delta + k^2 q(t)$ for $t \in \mR$. We wish to use \cite[Theorem VII.3.9 on p.\ 392]{kato1995perturbation} to show that the Neumann eigenvalues of $H(t)$ can be parametrized analytically with respect to $t$ (see \cite[Theorem 3.1, p.\ 442]{berezin2012schrodinger} and \cite[Theorem XII.13]{reed1978methods} for related results). In order to do this, we need to realize $H(z)$ with Neumann boundary values as a self-adjoint analytic family of unbounded operators on $L^2(\Omega)$. In the present case where $\Omega$ has Lipschitz boundary, the required results may be found in \cite[Section 2]{gesztesy2008robin} (in fact the easier abstract results in \cite[Appendix B]{gesztesy2008robin} would suffice).

Define the set 
\[
\mathcal{D} = \{ u \in H^1(\Omega) \,;\, \Delta u \in L^2(\Omega), \ \tilde{\gamma}_N u = 0 \text{ in $H^{-1/2}(\partial \Omega)$} \}
\]
where $\tilde{\gamma}_N$ is the weak Neumann trace operator in \cite[formula
(2.40) and (2.41)]{gesztesy2008robin}. We consider $H(z)$ as an unbounded linear operator on
$L^2(\Omega)$ with domain $\mathrm{dom}(H(z)) = \mathcal{D}$. The family
$(H(z))_{z \in \mC}$ has the following properties:
\begin{enumerate}
\item[(i)] 
Each $H(z)$ is closed and densely defined. This follows since $\Delta$ with domain $\mathcal{D}$ is self-adjoint by \cite[Theorem 2.6]{gesztesy2008robin}, hence $\Delta$ and consequently also $H(z)$ is closed and densely defined.
\item[(ii)] 
The family $(H(z))_{z \in \mC}$ is holomorphic of type (A) (see \cite[Section VII.2.1]{kato1995perturbation}) since for each $u \in \mathcal{D}$, the map 
\[
z \mapsto H(z)u = \Delta u + k^2 q_1 u + k^2 z(q_2-q_1) u
\]
is holomorphic.
\item[(iii)] 
The family $(H(z))_{z \in \mC}$ is a self-adjoint holomorphic family,
i.e. a holomorphic family of operators satisfying $H(z)^{*} = H(\bar z)$
(see \cite[Section
VII.3.1]{kato1995perturbation}): since $\Delta$ with domain $\mathcal{D}$ is
self-adjoint \cite[Theorem 2.6]{gesztesy2008robin} and the map $B(z): u \mapsto
k^2 (q_1 + z (q_2-q_1)) u$ is bounded on $L^2(\Omega)$, by \cite[Lemma
XII.1.6]{dunford1967linear} we have 
\begin{align*}
H(z)^* = \Delta^* + B(z)^* = \Delta + B(\bar{z}) = H(\bar{z}).
\end{align*}
\item[(iv)]
$H(z)$ has compact resolvent, when $z \in \mC$.
This can be seen as follows. Let 
$$
R_z(\zeta) := (H(z)-\zeta)^{-1},
$$
denote the resolvent. Arguing as in the proof of \cite[Corollary 2.7]{gesztesy2008robin} 
using \cite[Remark 2.19]{gesztesy2008robin}, we have that
$$
R_z(\zeta_0) := (H(z)-\zeta_0)^{-1},
$$
is compact, when $\zeta_0 \in \R_+$ is large enough. 
By the resolvent identity in \cite[Theorem VIII.2]{RSI} we have that
\begin{align*} 
R_z(\zeta) = R_z(\zeta_0) -(\zeta_0 - \zeta) R_z(\zeta)R_z(\zeta_0).
\end{align*}
The resolvent
$R_z(\zeta)$ is by definition continuous on $L^2(\Omega)$, when $\zeta$ is in the resolvent set.
The above formula implies hence that $R_z(\zeta)$ is compact on $L^2(\Omega)$, since 
$R_z(\zeta_0)$ is compact.

\end{enumerate}

Thus the family $(H(z))_{z \in \mC}$ satisfies the conditions in \cite[Theorem VII.3.9 on p.\ 392]{kato1995perturbation}, and there are real-analytic functions $\lambda_j(t)$ and real-analytic vector functions $\varphi_j(t)$, for $t \in [0,1]$, such that $\lambda_j(t)$ represent all the repeated eigenvalues of $H(t)$, $\varphi_j(t)$ are the corresponding eigenfunctions, and $(\varphi_j(t) )$ is an orthonormal basis of $L^2(\Omega)$. Since $\varphi_j(t) \in \mathcal{D}$, these are exactly the standard Neumann eigenvalues and eigenfunctions of $H(t)$ (see \cite[formula (2.41)]{gesztesy2008robin}). We may reorder $\lambda_j(t)$ and $\varphi_j(t)$ so that $\lambda_j(0) = \lambda_j$.

(b) We compute $\lambda_j'(t)$ using a variational formula: by \cite[formula
(VII.3.18), p.\ 391]{kato1995perturbation} and by the fact that $H'(t)u = k^2
(q_2-q_1) u$, we have 
\begin{equation} \label{lambdaj_derivative_formula}
\lambda_j'(t) = (H'(t) \varphi_j(t), \varphi_j(t))_{L^2(\Omega)} = \int_{\Omega} k^2(q_2-q_1) \varphi_j(t)^2 \,dx.
\end{equation}
Using the assumption that $q_2 \geq q_1$, we have $\lambda_j'(t) \geq 0$. Moreover, since $q_1 \not\equiv q_2$, we have $q_2 - q_1 \geq c > 0$ in some set $E$ of positive measure in $\Omega$. Thus we see that $\lambda_j'(t) > 0$ (otherwise if $\lambda_j'(t) = 0$, then $\int_E \varphi_j(t)^2 \,dx = 0$ which would contradict the unique continuation principle). This implies that each $\lambda_j(t)$ is a strictly increasing function on $[0,1]$.

(c) All other statements in (c) have been proved, except that $t \mapsto \varphi_j(t)$ is real-analytic as a $H^1(\Omega)$-valued function. To prove this, note first that for any $v \in L^2(\Omega)$ the map 
\[
t \mapsto (\varphi_j(t), v)_{L^2(\Omega)}
\]
is real-analytic on $[0,1]$. Now, if $\psi \in H^1(\Omega)$, we compute 
\begin{multline*}
(\varphi_j(t), \psi)_{H^1(\Omega)} = (\varphi_j(t), \psi)_{L^2(\Omega)} - (\Delta \varphi_j(t), \psi)_{L^2(\Omega)} \\
 = (\varphi_j(t), \psi)_{L^2(\Omega)} + k^2 (q(t) \varphi_j(t), \psi)_{L^2(\Omega)} - \lambda_j(t) (\varphi_j(t), \psi)_{L^2(\Omega)}.
\end{multline*}
Each term on the last line is real-analytic for $t \in [0,1]$. Thus $t \mapsto \varphi_j(t)$ is weakly, and hence strongly, analytic as a $H^1(\Omega)$-valued function.
\end{proof}

We will next combine Lemmas \ref{lemma_ndmap_general} and \ref{ev_pert} to obtain solutions of 
\begin{equation} \label{helmholtz_qt_equation}
(\Delta + k^2 q(t)) u_t = 0 \text{ in $\Omega$}, \qquad \p_{\nu} u_t|_{\p \Omega} = g
\end{equation}
that depend Lipschitz continuously on $t \in [0,1]$ as long as $g$ is orthogonal to a finite-dimensional subspace of $L^2(\p \Omega)$.

\begin{Lemma} \label{lemma_ut_properties}
Assume the conditions in Lemma \ref{ev_pert}. Let $t_1 < \ldots < t_K$ be the times when $0$ is a Neumann eigenvalue of $H(t)$, let $g \in L^2(\p \Omega)$, and let $u_t$ be the unique solution of \eqref{helmholtz_qt_equation} for $t \in [0,1] \setminus \{t_1, \ldots, t_K\}$.
\begin{enumerate}
\item[(a)] 
The map 
\[
[0,1] \setminus \{t_1, \ldots, t_K\} \to H^1(\Omega), \ \ t \mapsto u_t
\]
is real-analytic. The derivative $\partial_t u_t$ is the unique $H^1(\Omega)$ solution of $(\Delta + k^2 q(t)) v = -k^2 q'(t) u_t$ in $\Omega$ with $\p_{\nu} v|_{\p \Omega} = 0$. For any compact $F \subset [0,1] \setminus \{t_1, \ldots, t_K\}$, there is $C_F > 0$ such that 
\[
\norm{u_t}_{H^1(\Omega)} \leq C_F \norm{g}_{L^2(\p \Omega)}, \qquad t \in F.
\]
\item[(b)] 
Let $t_0$ be one of $t_1, \ldots, t_K$ and let $I = \{j \geq 1 \,;\, \lambda_j(t_0) = 0 \}$. Then 
\[
\norm{ u_t - \sum_{j \in I} (u_t, \varphi_j(t))_{L^2(\Omega)} \varphi_j(t) }_{H^1(\Omega)} \leq C \norm{g}_{L^2(\p \Omega)}
\]
uniformly for $t$ close to $t_0$.
\item[(c)] 
If $J = \{ j \geq 1 \,;\, \lambda_j(t_l) = 0 \text{ for some $l \in \{1, 2, \ldots, K \}$} \}$, then 
\[
\norm{ u_t - \sum_{j \in J} (u_t, \varphi_j(t))_{L^2(\Omega)} \varphi_j(t) }_{H^1(\Omega)} \leq C \norm{g}_{L^2(\p \Omega)}
\]
uniformly over $t \in [0,1] \setminus \{ t_1, \ldots, t_K \}$.
\item[(d)] 
With the notation of Lemma \ref{lemma_ndmap_general} let 
\[
D = D(q(t_1)) \oplus \ldots \oplus D(q(t_K)).
\]
If additionally $g \in D^{\perp}$, then $t \mapsto u_t$ extends uniquely as a Lipschitz continuous map 
\[
[0,1] \to H^1(\Omega), \ \ t \mapsto u_t.
\]
\end{enumerate}
\end{Lemma}
\begin{proof}
We first show that $0$ is a Neumann eigenvalue of $H(t)$ at only finitely many times $t$. Note that $\Delta + k^2 q_2$ has at most finitely many positive Neumann eigenvalues. Since the functions $\lambda_j(t)$ are strictly increasing and since the Neumann eigenvalues of $\Delta + k^2 q_2$ are given by $\lambda_j(1)$, we see that only finitely many of the functions $\lambda_j(t)$ have a zero in $[0,1]$. Thus there are only finitely many times $0 \leq t_1 < \ldots < t_K \leq 1$ in the interval $t \in [0,1]$ so that $0$ is a Neumann eigenvalue of $H(t)$.

\vspace{10pt}

{\it Proof of part (a).} \\ 

\noindent 
Fix any $g \in L^2(\p \Omega)$ and let $t \in [0,1] \setminus \{ t_1, \ldots,
t_K \}$. Then the Neumann problem for $\Delta + k^2 q(t)$ in $\Omega$ is
well-posed, and we define $u_t = T_{q(t)}(0,g)$ as the unique solution of
\eqref{helmholtz_qt_equation}. Fix $t_0 \in [0,1] \setminus \{t_1, \ldots,
t_K\}$, and note that for $t$ close to $t_0$ one has 
\[
(\Delta+k^2 q(t_0)) (u_t - u_{t_0}) = -k^2(q(t)-q(t_0)) u_t = -k^2(t-t_0)(q_2-q_1) u_t
\]
and $\p_{\nu}(u_t-u_{t_0})|_{\p \Omega} = 0$. It follows that 
\[
[\mathrm{Id} + k^2 (t-t_0) G(q(t_0)) ((q_2-q_1) \,\cdot\,)]u_t = u_{t_0}
\]
where $G(q(t_0)): F \mapsto T_{q(t_0)}(F,0)$ is bounded $L^2(\Omega) \to
H^1(\Omega)$ by Lemma \ref{lemma_ndmap_general}. Choosing $t$ close to $t_0$,
we can solve the last equation by Neumann series so that 
\[
u_t = \sum_{j=0}^{\infty}  (-k^{2})^j (t-t_0)^j [ G(q(t_0)) ((q_2-q_1)\,\cdot\,)]^j u_{t_0}.
\]
Thus $t \mapsto u_t$ is real-analytic in $[0,1] \setminus \{t_1, \ldots,
t_K\}$. Moreover, for any $t_0 \in [0,1] \setminus \{t_1, \ldots, t_K\}$ there
is $\eps(t_0) > 0$ so that 
\begin{equation} \label{ut_tzero_bound}
\norm{u_{t}}_{H^1(\Omega)} \leq 2 \norm{u_{t_0}}_{H^1(\Omega)} \leq C(t_0) \norm{g}_{L^2(\p \Omega)}, \qquad \abs{t-t_0} < \eps(t_0).
\end{equation}
This proves the uniform bound for $\norm{u_t}_{H^1(\Omega)}$ over any compact subset $F$ of $[0,1] \setminus \{ t_1, \ldots, t_K \}$. Finally, differentiating the power series and evaluating at $t_0$ yields 
\[
\partial_t u_t|_{t=t_0} = -k^2 G(q(t_0))((q_2-q_1) u_{t_0}),
\]
so that $\partial_t u_t|_{t=t_0}$ is the unique solution of $(\Delta + k^2 q(t_0)) v = -k^2(q_2-q_1) u_{t_0}$ in $\Omega$ with $\p_{\nu} v|_{\p \Omega} = 0$.

\vspace{10pt}

{\it Proof of part (b).} \\ 

\noindent 
Let $t_0$ be one of $t_1, \ldots, t_K$. Let $I = \{ j \geq 1 \,;\, \lambda_j(t_0) = 0 \}$ (so that $I$ is finite), and write 
\[
u_t = v_t + w_t, \qquad v_t = P_t u_t, \ \ w_t = Q_t u_t,
\]
where $P_t$ and $Q_t$ are the orthogonal projections on $L^2(\Omega)$ given by 
\[
P_t u = \sum_{j \in I} (u, \varphi_j(t))_{L^2(\Omega)} \varphi_j(t), \qquad Q_t u = \sum_{j \notin I} (u, \varphi_j(t))_{L^2(\Omega)} \varphi_j(t).
\]
We need to prove that $\norm{w_t}_{H^1(\Omega)} \leq C \norm{g}_{L^2(\p \Omega)}$ for $t$ close to $t_0$.

Fix some $\bar{t} \in [0,1] \setminus \{ t_1, \ldots, t_K \}$, and write $u_t = r + h_t$ where $r \in H^1(\Omega)$ is the unique solution of 
\[
(\Delta + k^2 q(\bar{t})) r = 0 \text{ in $\Omega$}, \qquad \p_{\nu} r|_{\p \Omega} = g.
\]
Then $h_t$ solves 
\[
(\Delta + k^2 q(t)) h_t = F_t \text{ in $\Omega$}, \qquad \p_{\nu} h_t|_{\p \Omega} = 0,
\]
where 
\[
F_t = -(\Delta + k^2 q(t)) r = -k^2(q(t)-q(\bar{t})) r = -k^2 (t-\bar{t}) (q_2-q_1) r.
\]
Now $w_t = Q_t u_t = Q_t r + Q_t h_t$, so that 
\[
\norm{w_t}_{L^2(\Omega)} \leq \norm{r}_{L^2(\Omega)} + \left[ \sum_{j \notin I} \left( \int_{\Omega} h_t \varphi_j(t) \,dx \right)^2 \right]^{1/2}.
\]
Testing the equation for $h_t$ against $\varphi_j(t)$ and integrating by parts gives that 
\[
\lambda_j(t) \int_{\Omega} h_t \varphi_j(t) \,dx = \int_{\Omega} F_t \varphi_j(t) \,dx
\]
and consequently 
\[
\norm{w_t}_{L^2(\Omega)} \leq \norm{r}_{L^2(\Omega)} + \left[ \sum_{j \notin I} \frac{1}{\lambda_j(t)^2} \left( \int_{\Omega} F_t \varphi_j(t) \,dx \right)^2 \right]^{1/2}.
\]
Now, the main point is that $\abs{\lambda_j(t_0)} \geq c > 0$ for $j \notin I$. Moreover, the formula \eqref{lambdaj_derivative_formula} implies that 
\[
0 \leq \lambda_j'(t) \leq C, \qquad \text{uniformly over $j \geq 1$ and $t \in [0,1]$}.
\]
These facts imply that there is $\eps > 0$ so that 
\[
\abs{\lambda_j(t)} \geq c/2, \qquad \text{uniformly over $j \notin I$ and $\abs{t-t_0} \leq \eps$}.
\]
It follows that 
\[
\norm{w_t}_{L^2(\Omega)} \leq \norm{r}_{L^2(\Omega)} + \frac{2}{c} \norm{F_t}_{L^2(\Omega)}, \qquad \abs{t-t_0} \leq \eps.
\]
Thus $\norm{w_t}_{L^2(\Omega)} \leq C \norm{g}_{L^2(\p \Omega)}$ uniformly over $t \in [t_0-\eps, t_0+\eps]$, since $\norm{r}_{H^1(\Omega)} \leq C \norm{g}_{L^2(\p \Omega)}$ and $F_t = -k^2 (t-\bar{t}) (q_2-q_1) r$.

Finally we estimate the $H^1$ norm. By Lemma \ref{lemma_ndmap_general}, we have  that
\begin{equation} \label{vt_definition}
v_t = -\sum_{j \in I} \frac{1}{\lambda_j(t)} \left[ \int_{\p \Omega} g \varphi_j(t) \,dS \right] \varphi_j(t).
\end{equation}
Now since $w_t$ solves the equation 
\[
(\Delta + k^2 q(t)) w_t = G_t \text{ in $\Omega$}, \qquad \p_{\nu} w_t|_{\p \Omega} = g,
\]
where by \eqref{vt_definition} 
\[
G_t = -(\Delta + k^2 q(t)) v_t = -\sum_{j \in I} \left[ \int_{\p \Omega} g \varphi_j(t) \,dS \right] \varphi_j(t),
\]
we obtain that 
\begin{align*}
\norm{\nabla w_t}_{L^2(\Omega)}^2 &= \int_{\Omega} (-\Delta w_t) w_t \,dx + \int_{\p \Omega} (\p_{\nu} w_t) w_t \,dS \\
 &= \int_{\Omega} (k^2 q(t) w_t - G_t) w_t \,dx + \int_{\p \Omega} g w_t \,dS.
\end{align*}
Consequently 
\[
\norm{w_t}_{H^1(\Omega)}^2 \leq C (\norm{w_t}_{L^2(\Omega)}^2 + \norm{G_t}_{L^2(\Omega)}^2) + \int_{\p \Omega} g w_t \,dS.
\]
Using Cauchy's inequality with $\eps$ in the boundary integral and the trace result $\norm{w_t}_{L^2(\p \Omega)} \leq C \norm{w_t}_{H^1(\Omega)}$, we obtain that 
\[
\norm{w_t}_{H^1(\Omega)}^2 \leq C (\norm{w_t}_{L^2(\Omega)}^2 + \norm{G_t}_{L^2(\Omega)}^2 + \norm{g}_{L^2(\p \Omega)}^2).
\]
We have seen above that $\norm{w_t}_{L^2(\Omega)} \leq C \norm{g}_{L^2(\p \Omega)}$ uniformly over $\abs{t-t_0} \leq \eps$, and the same is true for $\norm{G_t}_{L^2(\Omega)}$. Thus $\norm{w_t}_{H^1(\Omega)} \leq C \norm{g}_{L^2(\p \Omega)}$ uniformly over $\abs{t-t_0} \leq \eps$.

\vspace{10pt}

{\it Proof of part (c).} \\ 

\noindent This is completely analogous to the proof of part (b), upon using the fact that $\abs{\lambda_j(t)} \geq c > 0$ uniformly over $j \notin J$ and $t \in [0,1]$.

\vspace{10pt}

{\it Proof of part (d).} \\

\noindent 
Let now $g \in D^{\perp}$, let $t_0 = t_l$ where $1 \leq l \leq K$, and let $t \neq t_0$ be close to $t_0$. As in part (b), we write $u_t = v_t + w_t$ where $v_t = P_t u_t$ and $w_t = Q_t u_t$.

We first prove that under the assumption $g \in D^{\perp}$, the map $t \mapsto
v_t$ is a real-analytic from $[0,1]$ to $H^1(\Omega)$.
By \eqref{vt_definition} we have that
\begin{equation*} 
v_t = -\sum_{j \in I} \frac{1}{\lambda_j(t)} \left[ \int_{\p \Omega} g \varphi_j(t) \,dS \right] \varphi_j(t).
\end{equation*}
Now $\lambda_j(t_0) = 0$, so $v_t$ could potentially blow up as $t \to t_0$.
However, this is prevented by the fact that $g \in D^{\perp}$, which ensures
that $v_t$ may be written as 
\[
v_t = -\sum_{j \in I} \left[ \int_{\p \Omega} g \frac{\varphi_j(t)-\varphi_j(t_0)}{\lambda_j(t)-\lambda_j(t_0)} \,dS \right] \varphi_j(t).
\]
Since $\lambda_j$ and $\varphi_j$ are real-analytic, one has 
\begin{equation} \label{lambdaj_varphij_difference}
\lambda_j(t)-\lambda_j(t_0) = \mu_j(t)(t-t_0), \quad \varphi_j(t)-\varphi_j(t_0) = \psi_j(t)(t-t_0)
\end{equation}
where $\mu_j$ and $\psi_j$ are real-analytic near $t_0$ with $\psi_j$ taking values in $H^1(\Omega)$, and $\mu_j(t_0) = \lambda_j'(t_0) > 0$. Thus 
\begin{equation} \label{vt_representation}
v_t = -\sum_{j \in I} \left[ \int_{\p \Omega} g \frac{\psi_j(t)}{\mu_j(t)} \,dS \right] \varphi_j(t)
\end{equation}
where $\mu_j(t) \geq c > 0$ near $t_0$. The map $[0,1] \to L^2(\p \Omega)$, $t
\mapsto \psi_j(t)|_{\p \Omega}$ is real-analytic since the trace operator is
bounded from $H^1(\Omega)$ to $L^2(\p \Omega)$. Thus $t \mapsto v_t$ is
real-analytic near $t_0$, and one has $\norm{v_t}_{H^1(\Omega)} \leq C
\norm{g}_{L^2(\p \Omega)}$ near $t_0$. Combined with part (b), this implies
that 
\begin{equation} \label{ut_uniform_bound}
\norm{u_t}_{H^1(\Omega)} \leq C \norm{g}_{L^2(\p \Omega)} \quad \text{for $t$ close to $t_0$.}
\end{equation}

Next we define $u_{t_0}$ so that the map $s \mapsto u_s$ is Lipschitz continuous at $t_0$. Recalling the operator $D(q(t_0))^{\perp} \to H^1(\Omega), \ g \mapsto u_g$ where $u_g = T_{q(t_0)}(0,g)$ from Lemma \ref{lemma_ndmap_general}, we define 
\[
u_{t_0} := u_g + \sum_{j \in I} (v_{t_0} - u_g, \varphi_j(t_0))_{L^2(\Omega)} \varphi_j(t_0).
\]
Then $u_{t_0}$ solves $(\Delta + k^2 q(t_0)) u_{t_0} = 0$ in $\Omega$ with $\partial_{\nu} u_{t_0}|_{\p \Omega} = g$, and one has $P_{t_0} u_{t_0} = v_{t_0}$.

It remains to prove that $s \mapsto u_s$ is Lipschitz continuous at $t_0$. Note that 
\[
(\Delta+k^2 q(t_0)) (u_t - u_{t_0}) = -k^2(q(t)-q(t_0)) u_t = -k^2(t-t_0)(q_2-q_1) u_t
\]
and $\p_{\nu}(u_t-u_{t_0})|_{\p \Omega} = 0$. It follows from Lemma \ref{lemma_ndmap_general} that the function $-k^2(t-t_0)(q_2-q_1) u_t$ is in $N(q(t_0))^{\perp}$, and that 
\[
u_t - u_{t_0} = T_{q(t_0)}(-k^2(t-t_0)(q_2-q_1) u_t, 0) + \varphi
\]
for some $\varphi \in N(q(t_0))$. Define the operator 
\[
G(q(t_0)): N(q(t_0))^{\perp} \to H^1(\Omega), \ \ F \mapsto Q_{t_0} T_{q(t_0)}(F,0).
\]
Then $G(q(t_0))$ is bounded (since $Q_{t_0} = \mathrm{Id} - P_{t_0}$ is bounded on $H^1(\Omega)$), and 
\[
u_t - u_{t_0} = -k^2(t-t_0)G(q(t_0))((q_2-q_1)u_t) + P_{t_0}(u_t - u_{t_0}).
\]
Using the uniform bound \eqref{ut_uniform_bound}, we get that 
\[
\norm{u_t - u_{t_0}}_{H^1(\Omega)} \leq C \norm{g}_{L^2(\p \Omega)} \abs{t-t_0} + \norm{P_{t_0}(u_t-u_{t_0})}_{H^1(\Omega)}.
\]
To analyze the last term, we note that by the assumption that $g \in D^{\perp}$ and by \eqref{lambdaj_varphij_difference} and \eqref{vt_representation} 
{{\small 
\begin{align*}
 &P_{t_0} (u_t-u_{t_0}) = \sum_{j \in I} \left[ \int_{\Omega} (u_t - u_{t_0}) \varphi_j(t_0) \,dx \right] \varphi_j(t_0) \\
 &= \sum_{j \in I} \left[ \int_{\Omega} (u_t \varphi_j(t) - u_{t_0} \varphi_j(t_0)) \,dx \right] \varphi_j(t_0) - \sum_{j \in I} \left[ \int_{\Omega} u_t (\varphi_j(t) - \varphi_j(t_0)) \,dx \right] \varphi_j(t_0) \\
 &= -\sum_{j \in I} \left[ \int_{\p \Omega} g \left[ \frac{\psi_j(t)}{\mu_j(t)} - \frac{\psi_j(t_0)}{\mu_j(t_0)} \right] \,dS \right] \varphi_j(t_0) -  (t-t_0) \sum_{j \in I} \left[ \int_{\Omega} u_t \psi_j(t) \,dx \right] \varphi_j(t_0)
\end{align*}}}
$\!\!$where $\mu_j$ and $\psi_j$ are real-analytic near $t_0$ with $\psi_j$ taking values in $H^1(\Omega)$. Thus in particular 
\begin{align*}
\norm{P_{t_0}(u_t-u_{t_0})}_{H^1(\Omega)} \leq C (\norm{g}_{L^2(\p \Omega)} + \norm{u_t}_{L^2(\Omega)}) \abs{t-t_0}.
\end{align*}
Using \eqref{ut_uniform_bound} again, this concludes the proof that $s \mapsto u_s$ is Lipschitz continuous near $t_0$. Since this is true near $t_1, \ldots, t_K$, and since $s \mapsto u_s$ is real-analytic away from $\{ t_1, \ldots, t_K \}$, we have proved (d).
\end{proof}

We are now ready to prove Theorem \ref{thm_monotonicity_dimension}.

\begin{proof}[Proof of Theorem \ref{thm_monotonicity_dimension}]
We will do the proof in three steps.

\vspace{10pt}

%

{\it Step 1:} Definition of a finite-dimensional space $D$. \\

\noindent
We can assume that $q_2 \geq q_1$ and $q_2 \not\equiv q_1$, since the case $q_2
\equiv q_1$ is immediate.
Write $d_1 = d(q_1)$, $d_2 = d(q_2)$ and $N = d_2 - d_1$, and let $q(t)$ and
$\lambda_j(t)$ be as in Lemma \ref{ev_pert}. Now the positive Neumann
eigenvalues of $q(0)$ are $\lambda_1(0), \ldots,\lambda_{d_1}(0)$. Since the
functions $\lambda_j(t)$ are strictly increasing, the positive Neumann
eigenvalues related to $q(1)$ are  $\lambda_1(1), \ldots, \lambda_{d_1}(1),$$
\lambda_{j_1}(1), \ldots,$ $ \lambda_{j_{N}}(1)$ for some indices $j_1, \ldots,
j_{N}$ (here it is possible that $N = 0$). We reorder the indices for $j \geq
d_1+1$ so that the positive Neumann eigenvalues related to $q(1)$ are in descending order
$\lambda_1(1),\ldots, \lambda_{d_1}(1),$$ \lambda_{d_1+1}(1),\ldots, \lambda_{d_2}(1)$. It
follows that $\lambda_j(t)$ for $j \leq d_1$ are positive on $[0,1]$,
$\lambda_j(t)$ for $d_1+1 \leq j \leq d_2$ have a unique zero and cross from
negative to positive on $[0,1]$, and $\lambda_j(t)$ for $j \geq d_2+1$ are
always negative on $[0,1]$.

Let $0 < t_1 < \ldots < t_K < 1$ be the times when $0$ is a Neumann eigenvalue of $H(t)$, and let 
\[
D = D(q(t_1)) \oplus \ldots \oplus D(q(t_K))
\]
as in Lemma \ref{lemma_ut_properties}. By Lemmas \ref{lemma_ndmap_general} and
\ref{ev_pert}, $\dim(D(q(t_l)))$ is the multiplicity of $0$ as a Neumann
eigenvalue of $H(t_l)$, which is precisely the number of functions
$\lambda_j(t)$ that vanish at $t_l$. Since exactly $N = d_2-d_1$ functions
$\lambda_j$ have a zero in $[0,1]$, it follows that $\dim(D) \leq N$. (The
dimension of $D$ would be equal to $N$ if all the spaces $D(q(t_1)), \ldots,
D(q(t_K))$ would be linearly independent, but this may not be true in general.)

\vspace{10pt}

{\it Step 2:} We will next show that 
\[
((\Lambda(q_2) - \Lambda(q_1)) g, g) \geq 0 \qquad \text{for all $g \in D^{\perp}$}.
\]

\vspace{10pt}

\noindent Fix $g \in D^{\perp}$, and let $[0,1] \to H^1(\Omega), \ t \mapsto u_t$ be the map in Lemma \ref{lemma_ut_properties}. Since $q(0) = q_1$ and $q(1) = q_2$, it follows that 
\[
((\Lambda(q_2) - \Lambda(q_1)) g, g) = \int_{\p \Omega} u_1 g \,dS - \int_{\p \Omega} u_0 g \,dS.
\]
We write, for $t \in [0,1]$, 
\[
F(t) := \int_{\p \Omega} u_t g \,dS = \int_{\Omega} (\abs{\nabla u_t}^2 - k^2 q(t) u_t^2) \,dx.
\]
Then $F$ is Lipschitz continuous in $[0,1]$ since $t \mapsto u_t$ is:
\[
\abs{F(t) - F(t_0)} \leq \norm{u_t-u_{t_0}}_{L^2(\p \Omega)} \norm{g}_{L^2(\p \Omega)} \leq C_g \norm{u_t-u_{t_0}}_{H^1(\Omega)} \leq C_g \abs{t-t_0}.
\]

We compute the derivative of $F$ using the fact from Lemma \ref{lemma_ut_properties} that $u_t$ is real-analytic in $[0,1] \setminus \{t_1,\ldots,t_K\}$, and $\partial_t u_t$ is the unique solution of 
\[
(\Delta + k^2 q(t)) \partial_t u_t = - k^2 q'(t) u_t \text{ in $\Omega$},
    \qquad \partial_{\nu} \p_t u_t|_{\partial \Omega} = 0.
\]
Thus 
\begin{align*}
F'(t) &= 2 \int_{\Omega} (\nabla u_t \cdot \nabla \partial_t u_t - k^2 q(t) u_t
    \partial_t u_t) \,dx - k^2 \int_{\Omega} q'(t) u_t^2 \,dx \\
 &= 2 k^2 \int_{\Omega} q'(t) u_t^2 \,dx - k^2 \int_{\Omega} q'(t) u_t^2 \,dx \\
 &= k^2 \int_{\Omega} q'(t) u_t^2 \,dx.
\end{align*}
Since $F(t)$ is Lipschitz continuous and hence absolutely continuous, we may
use the fundamental theorem of calculus to compute 
\begin{align}
((\Lambda(q_2) - \Lambda(q_1)) g, g) &= F(1) - F(0) = \int_0^1 F'(t) \,dt \notag \\
 &= k^2 \int_0^1 \int_{\Omega} (q_2-q_1) u_t^2 \,dx \,dt. \label{ndmap_difference_computation}
\end{align}
Since $q_2 - q_1 \geq 0$ a.e., we get that $((\Lambda(q_2) - \Lambda(q_1)) g,g) \geq 0$ for $g \in D^{\perp}$ as required.

\vspace{10pt}

{\it Step 3:} One has $d(q_1, q_2) \leq d(q_2) - d(q_1)$. \\

\noindent By the previous step one has $((\Lambda(q_2) - \Lambda(q_1)) g, g) \geq 0$ for all $g \in D^{\perp}$. By \cite[Corollary 3.3]{harrach2018helmholtz} this implies that $\Lambda(q_1) \leq_{\dim(D)} \Lambda(q_2)$, i.e.\ that $\Lambda(q_2) - \Lambda(q_1)$ has $\leq \dim(D) \leq d(q_2) - d(q_1)$ negative eigenvalues.
\end{proof}

\section{Lower bounds for the number of negative eigenvalues} \label{sec_proof2}

In this section we will prove Theorems \ref{thm_monotonicity_dimension_optimal} and \ref{thm_monotonicity_dimension_example}. We will work under the assumption that $q_2-q_1$ is a positive constant, which ensures that the Neumann eigenvalues and eigenfunctions of $\Delta + k^2 q(t)$ behave in a very simple way as $t$ varies (in particular, analytic perturbation theory is not required).

\begin{proof}[Proof of Theorem \ref{thm_monotonicity_dimension_optimal}]
The proof proceeds in several steps.

\vspace{10pt}

{\it Step 1:} Notation for eigenvalues and eigenfunctions. \\

\noindent
Let $d_1 = d(q_1)$, and let 
\[
\lambda_1 \geq \lambda_2 \geq \ldots \geq \lambda_{d_1} > 0 > \lambda_{d_1+1} \geq \ldots \to -\infty
\]
be the Neumann eigenvalues of $\Delta + k^2 q_1$ in $\Omega$. (Here it is possible that $d_1 = 0$, and all eigenvalues are negative.) Let $(\varphi_j)_{j=1}^{\infty}$ be a corresponding orthonormal basis of $L^2(\Omega)$ consisting of Neumann eigenfunctions, i.e. 
\[
(\Delta + k^2 q_1) \varphi_j = \lambda_j \varphi_j \text{ in $\Omega$}, \qquad \partial_{\nu} \varphi_j|_{\partial \Omega} = 0.
\]
Define the potentials $q(t) = (1-t) q_1 + t q_2$. Since by assumption $c = q_2 - q_1$ is a positive constant, we have 
\[
q(t) = q_1 + tc.
\]
Now, one has 
\[
(\Delta + k^2 q_1) \varphi = \lambda \varphi \quad \Longleftrightarrow \quad (\Delta + k^2 q(t)) \varphi = (\lambda + k^2 c t) \varphi.
\]
Thus the Neumann eigenvalues of $\Delta + k^2 q(t)$ are given by 
\begin{align} \label{eq_lj_lin}
\lambda_j(t) = \lambda_j + k^2 ct,
\end{align}
and the corresponding $L^2$-orthonormal Neumann eigenfunctions $\varphi_j(t) = \varphi_j$ are independent of $t$. We note that the functions $\lambda_j(t)$, $t \in [0,1]$, are strictly increasing. They are positive if $j \leq d_1$, cross from negative to positive and satisfy $\lambda_j(t_j) = 0$ at times 
\[
0 < t_{d_1+1} \leq \ldots \leq t_{d_2} < 1
\]
if $d_1+1 \leq j \leq d_2$, and stay negative if $j \geq d_2+1$. Here $d_2 = d(q_2)$.

\vspace{10pt}

{\it Step 2:} Formula for $((\Lambda(q_2) - \Lambda(q_1)) g, g)$. \\

\noindent Fix $g \in L^2(\partial \Omega)$, and let $t \in [0,t_{d_1+1}) \cup (t_{d_2},1]$. Let $u_t$ be the solution of
\begin{equation*} 
(\Delta + k^2 q(t)) u_t = 0 \text{ in $\Omega$}, \qquad \partial_{\nu} u_t|_{\partial \Omega} = g.
\end{equation*}
Note that the Neumann problem is well-posed for $t$ in this range, and as in Lemma \ref{lemma_ndmap_general} one has the $L^2(\Omega)$-convergent representation 
\[
u_t = \sum_{j=1}^{\infty} c_j(t) \varphi_j
\]
with 
\begin{equation} \label{cj_formula}
c_j(t) = \int_{\Omega} u_t \varphi_j \,dx = -\frac{1}{\lambda_j(t)} \int_{\partial \Omega} g \varphi_j \,dS.
\end{equation}
As in Lemma \ref{lemma_ut_properties} (but with slightly different notation), we write $u_t = v_t + w_t$ where 
\begin{equation} \label{vt_wt_definition}
v_t = \sum_{j=d_1+1}^{d_2} c_j(t) \varphi_j, \qquad w_t = \sum_{j \notin [d_1+1,d_2]} c_j(t) \varphi_j.
\end{equation}
Thus we have 
\begin{align}
 &((\Lambda(q_2) - \Lambda(q_1)) g, g) = \int_{\partial \Omega} (u_1 - u_0) g \,dS \notag \\
 & \hspace{20pt} =  \sum_{j=d_1+1}^{d_2} (c_j(1)-c_j(0)) \int_{\partial \Omega} g \varphi_j \,dS + \int_{\partial \Omega} (w_1-w_0) g \,dS \notag \\
 & \hspace{20pt} = \sum_{j=d_1+1}^{d_2} \frac{k^2 c}{\lambda_j(1) \lambda_j(0)} \left( \int_{\partial \Omega} g \varphi_j \,dS \right)^2 + \int_{\partial \Omega} (w_1 - w_0) g \,dS. \label{quadratic_representation}
\end{align}
Note that the coefficient $\frac{k^2 c}{\lambda_j(1) \lambda_j(0)}$ is negative exactly when $d_1+1 \leq j \leq d_2$, so that the sum in \eqref{quadratic_representation} is $\leq 0$ while the last integral may be positive. 

\vspace{10pt}

{\it Step 3:} Formula for $((\Lambda(q_2-b) - \Lambda(q_1+a)) g, g)$. \\

\noindent We will now replace $q_1$ by $q_1 + a$ and $q_2$ by $q_2 - b$ and show that for suitable choices of $a$ and $b$, the negative contributions in \eqref{quadratic_representation} dominate the positive ones. This will imply that the corresponding quadratic form is negative on some finite-dimensional space, yielding a lower bound for the number of negative eigenvalues. We do the rescalings  
\[
a = c \alpha, \qquad b = c (1-\beta),
\]
where $\alpha, \beta \in [0,1]$ and 
\[
q(\alpha) = q_1 + a, \qquad q(\beta) = q_2 - b.
\]
The equation \eqref{quadratic_representation} now becomes 
\begin{multline}
((\Lambda(q_2-b) - \Lambda(q_1+a)) g, g) = \sum_{j=d_1+1}^{d_2} \frac{k^2 c (\beta-\alpha)}{\lambda_j(\beta) \lambda_j(\alpha)} \left( \int_{\partial \Omega} g \varphi_j \,dS \right)^2 \\
 + \int_{\partial \Omega} (w_{\beta} - w_{\alpha}) g \,dS. \label{quadratic_representation_rescaled}
\end{multline}
In the notation of Theorem \ref{thm_monotonicity_dimension_optimal}, one has $\mu_1 = \lambda_{d_1+1}$ and $\mu_2 = \lambda_{d_2}(1) = \lambda_{d_2} + k^2 c$. Then $t_{d_1+1} = c^{-1} k^{-2} |\mu_1|$ (since $\lambda_{d_1+1}(t_{d_1+1}) = 0$) and $t_{d_2} = 1 - c^{-1} k^{-2} \mu_2$ (since $\lambda_{d_2}(t_{d_2}) = 0$). It follows that  
\begin{equation} \label{ab_alphabeta_rescaling}
\left\{ \begin{split}
a \in [0,k^{-2}|\mu_1|) &\text{ if and only if } \alpha \in [0, t_{d_1+1}), \\
b \in [0,k^{-2} \mu_2) &\text{ if and only if } \beta \in (t_{d_2}, 1].
\end{split} \right.
\end{equation}
The next step is to show that the last integral in \eqref{quadratic_representation_rescaled} is uniformly bounded over $\alpha \in [0,t_{d_1+1})$ and $\beta \in (t_{d_2},1]$. This will follow since $w_t$ is related only to those eigenfrequencies that are uniformly bounded away from zero.

\vspace{10pt}

{\it Step 4:} $\norm{w_t}_{H^1(\Omega)} \leq C \norm{g}_{L^2(\p \Omega)}$ uniformly over $t \in[0,t_{d_1+1}) \cup (t_{d_2},1]$. \\

\noindent This follows directly from Lemma \ref{lemma_ut_properties}(c).

\vspace{10pt}

{\it Step 5:} Proof of part (a). \\

\noindent We will show that there is a subspace $V$ of $L^2(\p \Omega)$ with $\dim(V) = N_1$ such that \eqref{quadratic_representation_rescaled} is negative when $g \in V \setminus \{0\}$, $\alpha < t_{d_1+1}$ is close to $t_{d_1+1}$, and $\beta \in (t_{d_2}, 1]$. Combined with \eqref{ab_alphabeta_rescaling} and \cite[Lemma 3.2(b)]{harrach2018helmholtz} applied to $A = -(\Lambda(q_2-b) - \Lambda(q_1+a))$ with $r = 0$, this will prove part (a).

By the trace theorem and Step 4, we have 
\[
\left\lvert \int_{\partial \Omega} (w_{\beta} - w_{\alpha}) g \,dS \right\rvert \leq C \norm{w_{\beta}-w_{\alpha}}_{H^1(\Omega)} \norm{g}_{L^2(\p \Omega)} \leq C \norm{g}_{L^2(\p \Omega)}^2
\]
uniformly over $\alpha \in [0,t_{d_1+1})$ and $\beta \in  (t_{d_2},1]$. Thus 
\[
((\Lambda(q_2-b) - \Lambda(q_1+a)) g, g) \leq \sum_{j=d_1+1}^{d_2} \frac{k^2 c (\beta-\alpha)}{\lambda_j(\beta) \lambda_j(\alpha)} \left( \int_{\partial \Omega} g \varphi_j \,dS \right)^2 + C \norm{g}_{L^2(\p \Omega)}^2.
\]
If $j \in [d_1+1, d_2]$, then $\alpha < t_j$ and $\beta > t_j$, and $\lambda_j(\beta) = \lambda_j(t_j) + k^2 c (\beta-t_j) = k^2 c(\beta-t_j) > 0$. Thus one has 
\[
\frac{k^2 c(\beta-\alpha)}{\lambda_j(\beta)} \geq \frac{k^2 c(\beta-t_j)}{k^2 c(\beta-t_j)} = 1.
\]
Since $\lambda_j(\alpha) = \lambda_j(t_j) - k^2 c(t_j-\alpha) = - k^2 c(t_j-\alpha) < 0$, we obtain that 
\begin{multline*}
((\Lambda(q_2-b) - \Lambda(q_1+a)) g, g) \leq - \sum_{j=d_1+1}^{d_2} \frac{1}{k^2 c(t_j-\alpha)} \left( \int_{\partial \Omega} g \varphi_j \,dS \right)^2 \\
 + C \norm{g}_{L^2(\p \Omega)}^2
\end{multline*}
uniformly over $\alpha \in [0,t_{d_1+1})$ and $\beta \in  (t_{d_2},1]$.

Recall now the assumption that $\lambda_{d_1+1}$ has multiplicity $N_1$, and define 
\[
V = \mathrm{span}\{\varphi_{d_1+1}|_{\p \Omega}, \ldots, \varphi_{d_1+N_1}|_{\p \Omega} \}.
\]
Here $\varphi_{d_1+1}, \ldots, \varphi_{d_1+N_1}$ are Neumann eigenfunctions corresponding to $\lambda_{d_1+1}$. We claim that $\dim(V) = N_1$. For if $a_1 \varphi_{d_1+1}|_{\p \Omega} + \ldots + a_{N_1} \varphi_{d_1+N_1}|_{\p \Omega} = 0$, then the function $\varphi = a_1 \varphi_{d_1+1} + \ldots + a_{N_1} \varphi_{d_1+N_1}$ satisfies 
\[
(\Delta + k^2 q_1) \varphi = \lambda_{d_1+1} \varphi \text{ in $\Omega$}, \qquad \varphi|_{\p \Omega} = \p_{\nu} \varphi|_{\p \Omega} = 0.
\]
By the unique continuation principle this implies that $\varphi \equiv 0$, and since $\{\varphi_j \}$ are orthonormal in $L^2(\Omega)$ we obtain $a_1 = \ldots = a_{N_1} = 0$. This proves that $\dim(V) = N_1$.

Let now $g \in V \setminus \{0\}$. Since $\lambda_{d_1+1}$ has multiplicity
$N_1$ and since $t_j$ is the unique zero of $t \mapsto \lambda_j(t)$, by \eqref{eq_lj_lin} one has $t_{d_1+1} = \ldots = t_{d_1+N_1}$, and thus 
\begin{multline*}
((\Lambda(q_2-b) - \Lambda(q_1+a)) g, g) \leq - \frac{1}{k^2 c(t_{d_1+1}-\alpha)} \sum_{j=d_1+1}^{d_1+N_1} \left( \int_{\partial \Omega} g \varphi_j \,dS \right)^2 \\
 - \sum_{j=d_1+N_1+1}^{d_2} \frac{1}{k^2 c(t_j-\alpha)} \left( \int_{\partial \Omega} g \varphi_j \,dS \right)^2 + C \norm{g}_{L^2(\p \Omega)}^2
\end{multline*}
uniformly over $\alpha \in [0,t_{d_1+1})$ and $\beta \in  (t_{d_2},1]$. The middle term on the right is $\leq 0$, and writing 
\[
\delta = \inf_{g \in V, \norm{g}_{L^2(\p \Omega)} = 1} \sum_{j=d_1+1}^{d_1+N_1} \left( \int_{\partial \Omega} g \varphi_j \,dS \right)^2
\]
where $\delta > 0$ (the infimum is over the unit sphere in a finite dimensional normed space and the quantity inside the infimum is positive for $g \in V \setminus \{0\}$), we obtain that 
\[
((\Lambda(q_2-b) - \Lambda(q_1+a)) g, g) \leq \left( - \frac{\delta}{k^2 c(t_{d_1+1}-\alpha)} + C \right) \norm{g}_{L^2(\p \Omega)}^2
\]
where $C$ is uniform over $\alpha \in [0,t_{d_1+1})$ and $\beta \in  (t_{d_2},1]$. Thus choosing $\alpha < t_{d_1+1}$ sufficiently close to $t_{d_1+1}$, one has $((\Lambda(q_2-b) - \Lambda(q_1+a)) g, g) < 0$ for $g \in V \setminus \{0\}$. This concludes the proof of part (a).

\vspace{10pt}

{\it Step 6:} Proof of part (b). \\

\noindent This is completely analogous to Step 5: one defines the subspace 
\[
V = \mathrm{span}\{\varphi_{d_2-N_2+1}|_{\p \Omega}, \ldots, \varphi_{d_2}|_{\p \Omega} \}
\]
and shows that $((\Lambda(q_2-b) - \Lambda(q_1+a)) g, g) < 0$ for $g \in V \setminus \{0\}$ when $\beta \in (t_{d_2},1]$ is sufficiently close to $t_{d_2}$.
\end{proof}

\begin{proof}[Proof of Theorem \ref{thm_monotonicity_dimension_example}]
(a) If $\lambda_j$ are the Neumann eigenvalues of $\Delta + k^2 q$ in $\Omega$, then $\lambda_j \pm k^2 \eps$ are the eigenvalues of $\Delta + k^2 (q \pm \eps)$ in $\Omega$. Thus if $\eps_0 > 0$ is small enough and $\eps \leq \eps_0$, one has $d(q+\eps) - d(q-\eps) = N$, and $\Lambda(q+\eps)-\Lambda(q-\eps)$ has at most $N$ negative eigenvalues by Theorem \ref{thm_monotonicity_dimension}. Moreover, by Theorem \ref{thm_monotonicity_dimension_optimal} with $q_1 = q-\eps_0$, $q_2 = q+\eps_0$, $\mu_1 = -k^2 \eps_0$ and $\mu_2 = k^2 \eps_0$, we obtain that $\Lambda(q+\eps)-\Lambda(q-\eps)$ has at least $N$ negative eigenvalues for $\eps$ small.

(b) Recall that we now assume that $\Omega := (0,1)^2$.
It is enough to show that for any even $N \geq 2$, there is an eigenvalue
$\mu$ of $\Delta$ in $\Omega$ with multiplicity $N$. If this holds, then
choosing $c = -k^{-2} \mu$ gives that $0$ is an eigenvalue of $\Delta + k^2 c$
of multiplicity $N$, and the result follows from part (a).

An orthonormal basis of $L^2(\Omega)$ consisting of Neumann eigenfunctions of $\Delta$ in $\Omega$ is given by $(\varphi_{l_1,l_2})_{l_1, l_2=0}^{\infty}$, where 
\[
\varphi_{l_1,l_2}(x) = c_{l_1,l_2} \cos(l_1 x_1) \cos(l_2 x_2)
\]
for some normalizing constants $c_{l_1,l_2}$. The eigenvalue corresponding to $\varphi_{l_1,l_2}$ is $-(l_1^2+l_2^2)$. See e.g.\ \cite{grebenkov2013geometrical}.

We set $\lambda = 5^r$ where $r \geq 1$ is an odd integer, and write $N = r+1$. Since $5$ is a prime of the form $4m+1$, there are $4 N$ pairs $(s_1, s_2) \in \mZ^2$ such that $\lambda = s_1^2 + s_2^2$ \cite[Theorem 278]{hardy2008introduction}. Now $r$ is odd, so $\lambda$ is not a square and both $s_1$ and $s_2$ must be nonzero, and thus there are exactly $N$ pairs $(l_1,l_2) \in (\mathbb{N} \cup \{0\})^2$ so that $\lambda = l_1^2+l_2^2$. This shows that the multiplicity of $-\lambda$ as a Neumann eigenvalue of $\Delta$ is exactly $N$.
\end{proof}

\section{The Helmholtz equation with constant parameter} \label{sec_proof_constant}

In this section, we will treat the Neumann problem for the Helmholtz equation
\begin{equation} \label{neumann_problem_const}
\left\{ \begin{array}{rcl} (\Delta + k^2 q) u \!\!\!&=&\!\!\! 0 \text{ in $\Omega$}, \\[5pt]
\partial_{\nu} u \!\!\!&=&\!\!\! g \text{ on $\partial \Omega$}. \end{array} \right.
\end{equation}
with a constant coefficient $q(x)=\mathrm{const}$.
In this case, the Helmholtz solution operator can be expressed using the Neumann eigenfunctions 
of the Laplace equation, which allows us to give a simple independent proof of Theorem~\ref{thm_monotonicity_dimension}, and show that the 
dimension bound in Theorem~\ref{thm_monotonicity_dimension} is sharp for the Helmholtz solution operators.

For the special case of a constant coefficient in a two-dimensional unit square we also derive an infinite matrix representation of the Neumann-Dirichlet-operator
and study numerically the question whether the bound in Theorem~\ref{thm_monotonicity_dimension} is sharp for the Neumann-Dirichlet-operators.

\subsection{The dimension bound for the constant parameter case} 

\begin{Definition}
For $0\not\equiv q\in L^\infty(\Omega)$, and a non-resonant wavenumber $k>0$ we define the \emph{Helmholtz solution operator}
\begin{align*}
S(q):&\ H^1(\Omega)\to H^1(\Omega),\ F\mapsto v, 
\end{align*}
where $v\in H^1(\Omega)$ solves
\[
\int_\Omega \left(\nabla v\cdot \nabla w - k^2 q v w\right)\, dx =  (F,w)_{H^1(\Omega)}\quad \text{ for all } w\in H^1(\Omega).
\]
\end{Definition}

Note that, the Neumann-Dirichlet-operator
\[
\Lambda(q):\ L^2(\partial \Omega)\to L^2(\partial \Omega), \quad g\mapsto u|_{\partial \Omega},
\]
where $u\in H^1(\Omega)$ solves \eqref{neumann_problem_const}, obviously fulfills
\[
\Lambda(q)=\gamma S(q) \gamma^*.
\]
where $\gamma$ denotes the compact trace operator
\[
\gamma:\ H^1(\Omega) \to L^2(\partial \Omega),\quad v\mapsto v|_{\partial \Omega}.
\]

\begin{Theorem}\label{Thm:solution_operator_sharp_bound}
Let $\Omega\subset \R^n$ be a bounded Lipschitz domain. Let $a,b\in \R$ with $a<b$, and let $k>0$ be non-resonant for $q(x)=a$ and $q(x)=b$.
Then
\begin{enumerate}
\item[(a)] $S(b)-S(a)$ has exactly $d(b)-d(a)$ negative eigenvalues.
\item[(b)] $\Lambda(b)-\Lambda(a)$ has at most $d(b)-d(a)$ negative eigenvalues.
\end{enumerate}
\end{Theorem}

Note that (b) follows from Theorem~\ref{thm_monotonicity_dimension}, but our proof of Theorem~\ref{Thm:solution_operator_sharp_bound} is independent of this result and rather elementary, so we believe that this is of independent interest.

As in the proof of Lemma~\ref{lemma_ndmap_general} (see also \cite[Section 2.1]{harrach2018helmholtz}), we let 
$\mathrm{Id}:\ H^1(\Omega) \to H^1(\Omega)$ denote the identity operator, $\iota:\ H^1(\Omega)\to L^2(\Omega)$ denote the compact embedding,
and $M_{q}:\ L^2(\Omega)\to L^2(\Omega)$ denote the multiplication operator by $q$. Then
$K:=\iota^* \iota$ and $K_q:=\iota^* M_{q} \iota$ are compact self-adjoint linear operators from $H^1(\Omega)$ to $H^1(\Omega)$, 
and 
\[
S(q):=(\mathrm{Id}-K-k^2 K_q)^{-1},
\]
where the inverse exists if and only if $k>0$ is non-resonant for the potential $q$, cf., e.g., \cite[Lemma~2.2]{harrach2018helmholtz}.

For constant coefficients $q(x)=a\in \R$ this simplifies to
\[
K_a=aK \quad \text{ and } \quad S(a)=(\mathrm{Id}-(1+a k^2) K)^{-1}.
\]

Since $K:\ H^1(\Omega)\to H^1(\Omega)$ is a compact self-adjoint, positive
definite operator, there exists an orthonormal basis $(v_l)_{l\in \mN}$
of $H^1(\Omega)$ of eigenfunctions corresponding to eigenvalues $\lambda_l>0$,
\[
Kv_l = \lambda_l v_l \quad \text{ for all } l\in \mN.
\]
Note that in this section, $\lambda_l$ are the eigenvalues of the compact operator $K$ which converge to zero (unlike in the earlier sections, where $\lambda_j$ were Neumann eigenvalues converging to $-\infty$).

\begin{Lemma}\label{lemma:sharp_bounds_proof}
\begin{enumerate}
\item [(a)] A function $v\in H^1(\Omega)$ is an eigenfunction of $K$ with eigenvalue $\lambda$ if and only if
$v$ is a Neumann eigenfunction of the Laplace equation with Neumann eigenvalue $\frac{1}{\lambda}-1$, i.e.,
\[
-\Delta v=\left( \frac{1}{\lambda} -1 \right) v,\quad \partial_\nu v|_{\partial \Omega}=0.
\]
\item[(b)] $k>0$ is non-resonant for the potential $q(x)=a\in \R$ if and only if
\[
\frac{1}{1+a k^2} \not\in \{\lambda_1,\lambda_2,\ldots \},
\]
i.e., if and only if $a k^2$ is not a Neumann eigenvalue. Moreover,
\[
d(a)=\# \left\{ \lambda_l:\ \lambda_l>\frac{1}{1+a k^2}\right\}.
\]
\item[(c)] If $k>0$ is non-resonant for the potential $q(x)=a\in \R$ then
\[
S(a)v_l=(\mathrm{Id}-(1+a k^2) K)^{-1} v_l = \frac{1}{1-(1+a k^2)\lambda_l}  v_l.
\]
\item[(d)] Let $a,b\in \R$ with $a<b$, and $k>0$ be non-resonant for $q(x)=a$ and $q(x)=b$. Then
\begin{align*}
\left( F, \left( S(b)-S(a) \right) F\right)_{H^1(\Omega)}
= \sum_{l=1}^\infty c_l \left(  F, v_l \right)_{H^1(\Omega)}^2,
\end{align*}
where $0\neq c_l\in \R$, and the number of negative $c_l$ is exactly $d(b)-d(a)$.
\end{enumerate}
\end{Lemma}
\begin{proof}
\begin{enumerate}
$Kv=\lambda v$ is equivalent to
\begin{align*}
\lefteqn{\int_\Omega \nabla v\cdot \nabla w\, dx - \left( \frac{1}{\lambda} -1 \right) \int_\Omega v w\, dx}\\
&=
\left( (\mathrm{Id} - K) v, w \right)_{H^1(\Omega)} - \left( \frac{1}{\lambda} -1 \right) \left( K v, w \right)_{H^1(\Omega)}\\
&= \left( (\mathrm{Id} -  \frac{1}{\lambda} K) v, w\right)_{H^1(\Omega)} = 0 \quad \text{ for all } w\in H^1(\Omega),
\end{align*}
which is the variational formulation equivalent to 
\[
-\Delta v=\left( \frac{1}{\lambda} -1 \right) v,\quad \partial_\nu v|_{\partial \Omega}=0.
\]
This proves (a). 

The first part of (b) and (c) are obvious. The second part of (b) has been proven in \cite[Lemma 2.1]{harrach2018helmholtz}.

To prove (d) note that for all $F\in H^1(\Omega)$
\[
F = \sum_{l=1}^\infty v_l \left( F, v_l \right)_{H^1(\Omega)},
\]
where the sum is convergent in $H^1(\Omega)$. Hence,
\begin{align*}
\lefteqn{\left( F, \left( S(b)-S(a) \right) F\right)_{H^1(\Omega)}}\\
&= \left( F, \left( S(b)-S(a) \right)  \sum_{l=1}^\infty v_l \left( F, v_l \right)_{H^1(\Omega)} \right)_{H^1(\Omega)}\\
&= \sum_{l=1}^\infty \left( \frac{1}{1-(1+b k^2)\lambda_l} - \frac{1}{1-(1+a k^2)\lambda_l}  \right) \left( F, v_l\right)_{H^1(\Omega)}^2.
\end{align*}
For the coefficients
\begin{align*}
c_l&:=\frac{1}{1-(1+b k^2)\lambda_l} - \frac{1}{1-(1+a k^2)\lambda_l}\\
&=\frac{(b-a)k^2 }{\left( \frac{1}{\lambda_l}-(1+b k^2) \right) \left(\frac{1}{\lambda_l}-(1+a k^2)\right)}
\end{align*}
we obviously have that $c_l\neq 0$ and that $c_l<0$ if and only if
\[
\frac{1}{1+a k^2} > \lambda_l > \frac{1}{1+b k^2}.
\]
By the second part of (b), the number of negative $c_l$ is exactly $d(b)-d(a)$.
\end{enumerate}
\end{proof}

\begin{proof}[Proof of Theorem~\ref{Thm:solution_operator_sharp_bound}.]
Using Lemma~\ref{lemma:sharp_bounds_proof} we have that
\[
V:=\mathrm{span}\{ v_l:\ c_l<0\}\subset H^1(\Omega),
\]
is a subspace of dimension $\mathrm{dim}(V)=d(b)-d(a)$, 
\begin{align*}
\left( F, \left( S(b)-S(a) \right) F\right)_{H^1(\Omega)} &< 0 \quad \text{ for all } F\in V \setminus \{0\}, \text{ and }\\
\left( F, \left( S(b)-S(a) \right) F\right)_{H^1(\Omega)} &\geq  0 \quad \text{ for all } F\perp_{H^1} V.
\end{align*}
Using, e.g., \cite[Lemma 3.2]{harrach2018helmholtz} this shows that $S(b)-S(a)$ has exactly $d(b)-d(a)$ negative eigenvalues
and thus proves Theorem~\ref{Thm:solution_operator_sharp_bound}(a).

Using that $\Lambda(b)-\Lambda(a)=\gamma (S(b)-S(a)) \gamma^*$, we also have that
\[
\int_{\partial \Omega} g \left( \Lambda(b)-\Lambda(a) \right) g\, dS \geq 0
\]
for all $g$ with $\gamma^* g\perp_{H^1} V$, which is equivalent to
\[
\left( g, \gamma v \right)_{L^2(\partial\Omega)}=\left( \gamma^* g, v \right)_{H^1(\Omega)} =0 \quad \text{ for all } v\in V,
\] 
and thus 
\[
\int_{\partial \Omega} g \left( \Lambda(b)-\Lambda(a) \right) g\, dS \geq 0 \quad \text{ for all } g\perp_{L^2} \gamma (V).
\]
Using $\mathrm{dim}(\gamma (V)) \leq \mathrm{dim}(V)=d(b)-d(a)$ and \cite[Cor.~3.3]{harrach2018helmholtz}), this shows that $\Lambda(b)-\Lambda(a)$ has 
at most $d(b)-d(a)$ negative eigenvalues and thus proves Theorem~\ref{Thm:solution_operator_sharp_bound}(b).
\end{proof}

\subsection{Helmholtz equation on the two-dimensional unit square}

We now consider the special case of the Helmholtz equation with constant parameter $q(x)=a\in \R$ on the two-dimensional unit square
\[
\Omega:=(0,1)^2
\]
and derive an infinite matrix representation for the Neumann-Dirichlet-operator $\Lambda(a)$.

For the unit square the Neumann eigenfunctions are well known:
\begin{Lemma}\label{lemma:Neumann_eigenfunctions_square}
For $l,m\in \mN_0$ we define 
\[
v_{l,m}:\Omega\to \R, \quad v_{l,m}(x,y):= d_l d_m \cos(\pi l x) \cos(\pi m y),
\]
with $d_0:=1$ and $d_j=\sqrt{2}$ for $j\in \mN$. The functions $v_{l,m}$ are Neumann eigenfunctions of the Laplacian
\begin{align*}
-\Delta v_{l,m}&=\pi ^2 (l^2 + m^2) v_{l,m}, \qquad \partial_\nu v_{l,m}|_{\partial \Omega}=0,
\end{align*}
and eigenfunctions of $K:\ H^1(\Omega)\to H^1(\Omega)$
\[
Kv_{l,m}=\lambda_{l,m} v_{l,m}\quad \text{ with } \quad \lambda_{l,m}:=\frac{1}{1+\pi^2 (l^2+m^2)}.
\]

$(v_{l,m})_{l,m\in \mN}$ is an orthonormal basis of $L^2(\Omega)$,
and $(\sqrt{\lambda_{l,m}} v_{l,m})_{l,m\in \mN}$ is an orthonormal basis of $H^1(\Omega)$.
\end{Lemma}
\begin{proof}
It is easily checked that the functions $(v_{l,m})_{l,m\in \mN}$ are Neumann eigenfunctions and that they form an orthonormal basis of $L^2(\Omega)$.
Lemma \ref{lemma:sharp_bounds_proof}, that the $(v_{l,m})_{l,m\in \mN}$ are also eigenfunctions of $K$, and this yields that
\[
\left( v_{l,m}, v_{l',m'} \right)_{H^1(\Omega)} = \frac{1}{\lambda_{l,m}} \int_{\Omega} v_{l,m} v_{l',m'}\, dx = \frac{1}{\lambda_{l,m}} \delta_{l,l'} \delta_{m,m'},
\]
which shows that $(\sqrt{\lambda_{l,m}} v_{l,m})_{l,m\in \mN}$ is an orthonormal basis of $H^1(\Omega)$.
\end{proof}

We can now expand the Neumann-Dirichlet operator $\Lambda(a)$ in an orthonormal basis of cosine functions on the four sides of $\partial \Omega$.
\begin{Lemma}\label{lemma:NtD_square_matrix}
Define $g_s:\ \partial \Omega\to \R$ ($s\in \mN_0$) by setting for all $j\in \mN_0$
\begin{align*}
g_{4j+0}(x,y)&:= d_j \cos ( j \pi y) \mathbf{1}_{\Sigma_\text{right}}(x,y), &  \Sigma_\text{right}&:=\{1\}\times (0,1),\\
g_{4j+1}(x,y)&:= d_j \cos ( j \pi (1-x))  \mathbf{1}_{\Sigma_\text{top}}(x,y), &  \Sigma_\text{top}&:=(0,1)\times \{1\},\\
g_{4j+2}(x,y)&:=d_j \cos ( j \pi (1-y)) \mathbf{1}_{\Sigma_\text{left}}(x,y),&  \Sigma_\text{left}&:=\{0\}\times (0,1),\\
g_{4j+3}(x,y)&:=d_j \cos ( j \pi x)  \mathbf{1}_{\Sigma_\text{bottom}}(x,y),&  \Sigma_\text{bottom}&:=(0,1)\times \{0\}.
\end{align*}
Then $(g_s)_{s\in \mN_0}\subseteq L^2(\partial \Omega)$ is an orthonormal basis of $L^2(\partial \Omega)$.
 
The infinite matrix representation of $\Lambda(a)$ with respect to this basis is given by $4\times 4$-blocks of the form 
\begin{align*}
\left(\int_{\partial \Omega} g_{4i+p}  \Lambda(a) g_{4j+r}\, ds\right)_{p,r=0,\ldots,3}
&=\left( \begin{array}{c c c c}
M_0 & M_1 & M_2 & M_3\\
M_3 & M_0 & M_1 & M_2\\
M_2 & M_3 & M_0 & M_1\\
M_1 & M_2 & M_3 & M_0
 \end{array}\right)
\end{align*}
where (for $i,j\in \mN_0$)
\begin{align*}
M_0&= \left\{ \begin{array}{r l} 
\delta_{ij} \frac{ \coth(\sqrt{\pi^2 i^2- a k^2})}{ \sqrt{\pi^2 i^2- a k^2}} & \text{for $i^2>\frac{a k^2}{\pi^2}$,}\\[+1ex]
-\delta_{ij} \frac{ \cot(\sqrt{ak^2-\pi^2 i^2})}{ \sqrt{ak^2-\pi^2 i^2}} & \text{for $i^2<\frac{a k^2}{\pi^2}$,}
\end{array}\right. &
M_1&=\frac{(-1)^i d_i d_j }{\pi^2 (i^2+j^2)- a k^2},\\[+1ex]
M_2& = \left\{ \begin{array}{r l} 
\delta_{ij} (-1)^i \frac{ \operatorname{csch}(\sqrt{\pi^2 i^2- a k^2})}{\sqrt{\pi^2 i^2- a k^2}} & \text{for $i^2>\frac{a k^2}{\pi^2}$,}\\[+1ex]
-\delta_{ij} (-1)^i \frac{  \csc( \sqrt{ak^2-\pi^2 i^2})}{ \sqrt{a k^2-\pi^2 i^2} } & \text{for $i^2<\frac{a k^2}{\pi^2}$,}
\end{array}\right. &
M_3&=\frac{(-1)^j d_i d_j }{\pi^2 (i^2+j^2)- a k^2}.
\end{align*}
\end{Lemma}
\begin{proof}
Clearly, $(g_s)_{s\in \mN_0}\subseteq L^2(\partial \Omega)$ is an orthonormal basis of $L^2(\partial \Omega)$, and 
\begin{align*}
I_0(j,l,m)&:=\int_{\partial \Omega} g_{4j+0}\, v_{l,m}|_{\partial \Omega}\, ds 
= (-1)^l d_l \delta_{jm},\\
I_1(j,l,m)&:=\int_{\partial \Omega} g_{4j+1}\, v_{l,m}|_{\partial \Omega}\, ds 
= (-1)^{m+j} d_m \delta_{jl},\\
I_2(j,l,m)&:=\int_{\partial \Omega} g_{4j+2}\, v_{l,m}|_{\partial \Omega}\, ds 
= (-1)^j d_l \delta_{jm},\\
I_3(j,l,m)&:=\int_{\partial \Omega} g_{4j+3}\, v_{l,m}|_{\partial \Omega}\, ds 
= d_m \delta_{jl}.
\end{align*}

Using that $(\sqrt{\lambda_{lm}} v_{lm})_{l,m\in \mN}$ is an orthonormal basis of $H^1(\Omega)$ that diagonalizes the solution operator
(cf.\ lemma~\ref{lemma:sharp_bounds_proof}(c))
we have that
\begin{align*}
\int_{\partial \Omega} g_{4i+p}  \Lambda(a) g_{4j+r}\, ds&= \left( \gamma^* g_{4i+p} , S(a) \gamma^* g_{4j+r} \right)_{H^1(\Omega)}\\
&= \sum_{l,m=0}^\infty \frac{\lambda_{lm}}{1-(1+ak^2)\lambda_{lm}} I_p(i,l,m) I_r(j,l,m) \\
&= \sum_{l,m=0}^\infty \frac{1}{\pi^2 (l^2+m^2)- a k^2} I_p(i,l,m) I_r(j,l,m).
\end{align*}
The assertion then follows from a simple calculation using the sum formulas 
\begin{align*}
\sum_{m=0}^\infty \frac{d_m^2}{\pi^2 m^2+c} &= \left\{ \begin{array}{r l}  -\frac{ \cot(\sqrt{-c})}{\sqrt{-c}} & \text{ for $c<0$, $\sqrt{-c}\not\in \pi \mN$,}\\
 \frac{  \coth(\sqrt{c})}{\sqrt{c} } & \text{ for $c>0$,}
\end{array}\right.\\
\sum_{m=0}^\infty (-1)^m \frac{d_m^2}{\pi^2 m^2+c} &= \left\{ \begin{array}{r l} -\frac{  \csc( \sqrt{-c})}{ \sqrt{-c} } & \text{ for $c<0$, $\sqrt{-c}\not\in \pi \mN$,}\\
 \frac{ \operatorname{csch}(\sqrt{c})}{\sqrt{c}} & \text{ for $c>0$,}
\end{array}\right.
\end{align*}
(see e.g. \cite[formulas (1) on p. 327 and (4) on p. 329]{remmert1991theory}).
\end{proof}

\subsection{Numerical evaluation of the dimension bound}

We still consider the special case of the Helmholtz equation on the unit square $\Omega=(0,1)^2$ with constant parameter $q(x)=a\in \R$, resp., $q(x)=b\in \R$,
and fix $k:=1$ without loss of generality. It follows from lemma~\ref{lemma:sharp_bounds_proof} and lemma \ref{lemma:Neumann_eigenfunctions_square} that
resonances occur when $a$ or $b$ equals $\pi^2 (l^2+m^2)$ with $l,m\in \mN_0$, and that
\begin{align*}
d(b)-d(a) 
&= \#\{ l,m\in \mN_0:\ a < \pi^2 (l^2+m^2)< b\}.
\end{align*}

We know from Theorem~\ref{Thm:solution_operator_sharp_bound} (and the more general Theorem~\ref{thm_monotonicity_dimension})
that $\Lambda(b)-\Lambda(a)$ will have at most $d(b)-d(a)$ negative eigenvalues. Moreover, we know
from Theorem~\ref{thm_monotonicity_dimension_example} that this bound is achieved, when $a$ and $b$ are sufficiently close together and only slightly smaller, resp., larger than 
a Neumann eigenvalue $\pi^2 (l^2+m^2)$, and $d(b)-d(a)$ will then be the multiplicity of this Neumann eigenvalue which can 
attain any even positive integer.

We will now numerically evaluate how the number of negative eigenvalues of $\Lambda(b)-\Lambda(a)$ behaves. For this end we use the numerical programming language Matlab to calculate 
a $1000\times 1000$ matrix approximating $\Lambda(b)-\Lambda(a)$ using the matrix representation formula in lemma \ref{lemma:NtD_square_matrix} for $i,j=0,\ldots, 249$.
We estimated the error in this finite dimensional approximation to be below $\delta=10^{-5}$ in the spectral norm by comparing $\Lambda(b)-\Lambda(a)$ to its upper left $500\times 500$ entries (filled up by zeros to a $1000\times 1000$ matrix). 
Accordingly, we considered eigenvalues below $-\delta$ to be negative and counted their number (with multiplicity).

Figure \ref{fig:Comparison_Plot} shows this numerically computed number of negative eigenvalues of $\Lambda(b)-\Lambda(a)$ and the theoretical bound $d(b)-d(a)$ as a function of $b\in \mR$ for $a=-10$ (top left), $a=10$ (top right), $a=100$ (bottom left), and $a=200$ (bottom right). Whenever $b$ crosses an eigenvalue $\pi^2 (l^2+m^2)$ with $l,m\in \mN_0$ the theoretical bound $d(b)-d(a)$ increases by the multiplicity of this eigenvalue. The plots indicate that the number of negative eigenvalue $\Lambda(b)-\Lambda(a)$ also increases by the 
multiplicity of this eigenvalue but that there is an additional effect decreasing the number of negative eigenvalues when $b-a$ increases.

\begin{figure}
\begin{center}
\begin{tabular}{c c}
\mbox{\includegraphics[width=5.5cm]{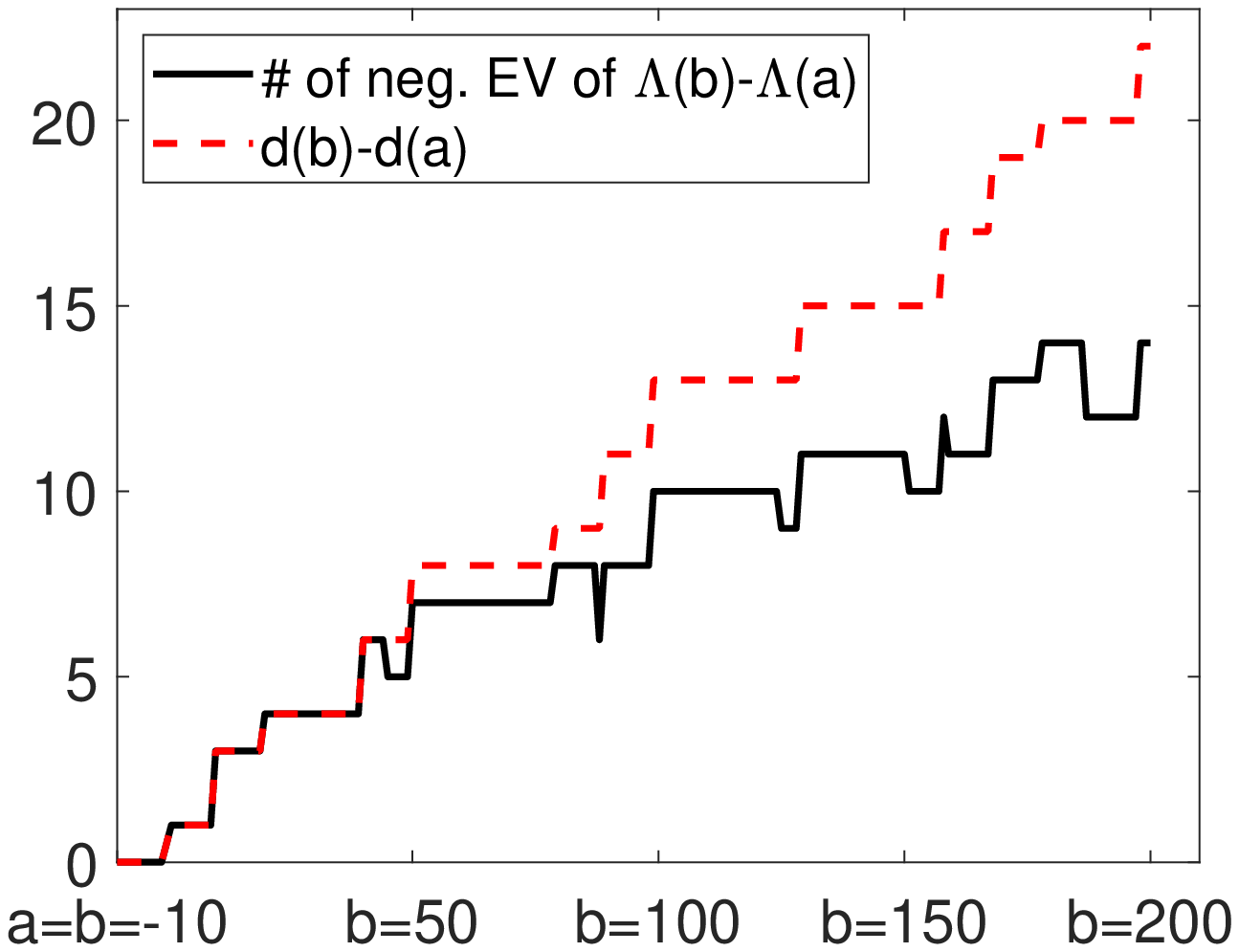}} &
\mbox{\includegraphics[width=5.5cm]{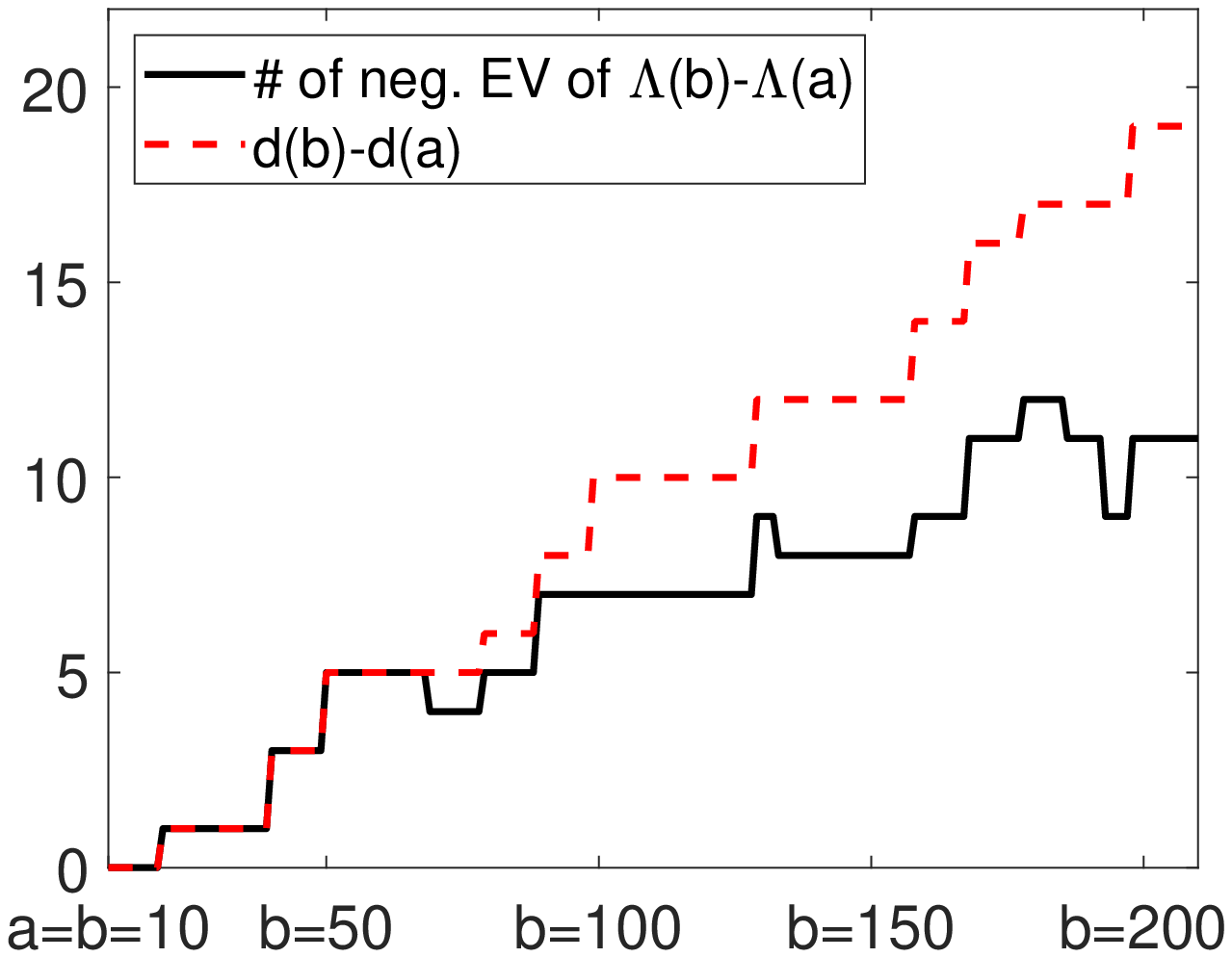}}\\
\mbox{\includegraphics[width=5.5cm]{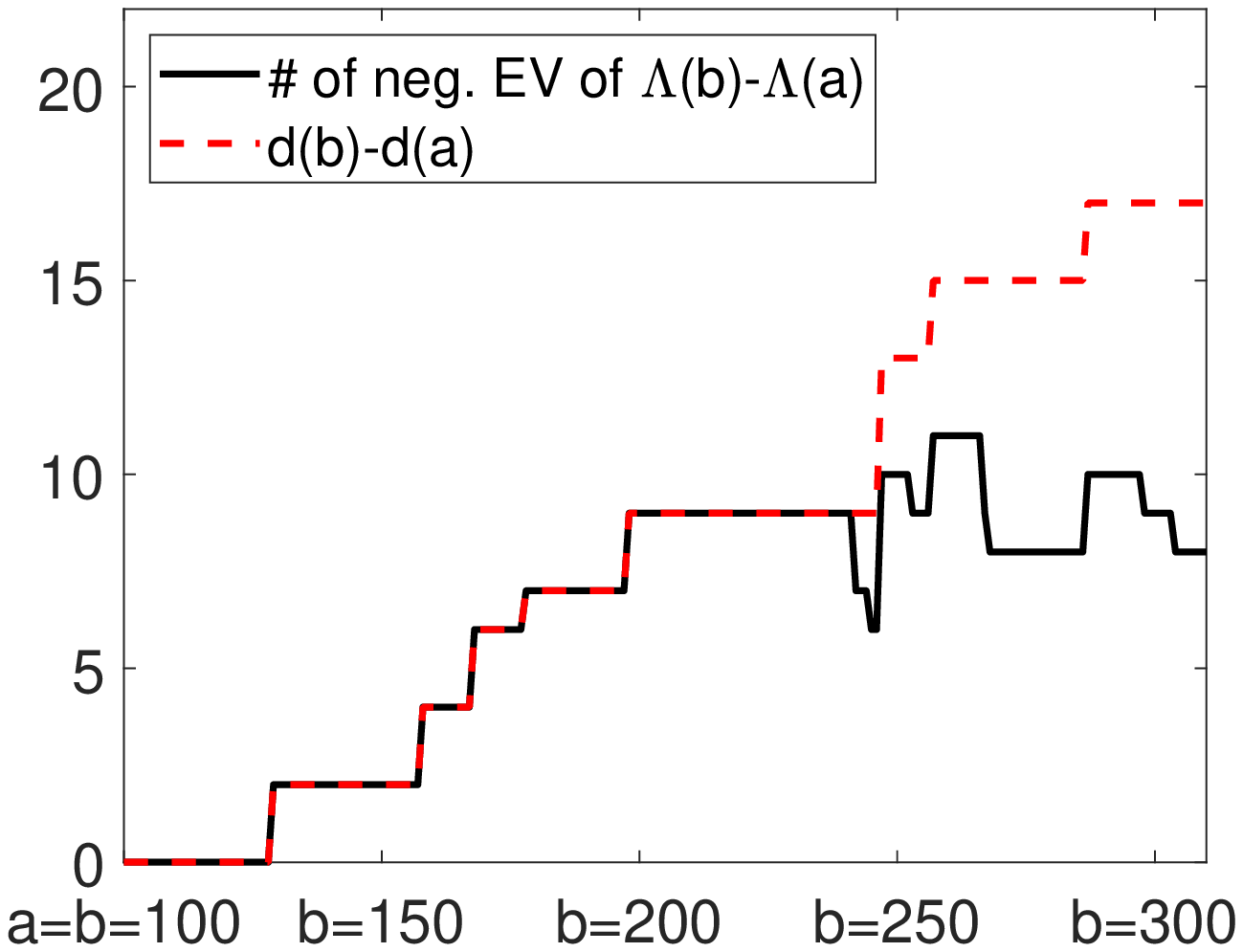}} &
\mbox{\includegraphics[width=5.5cm]{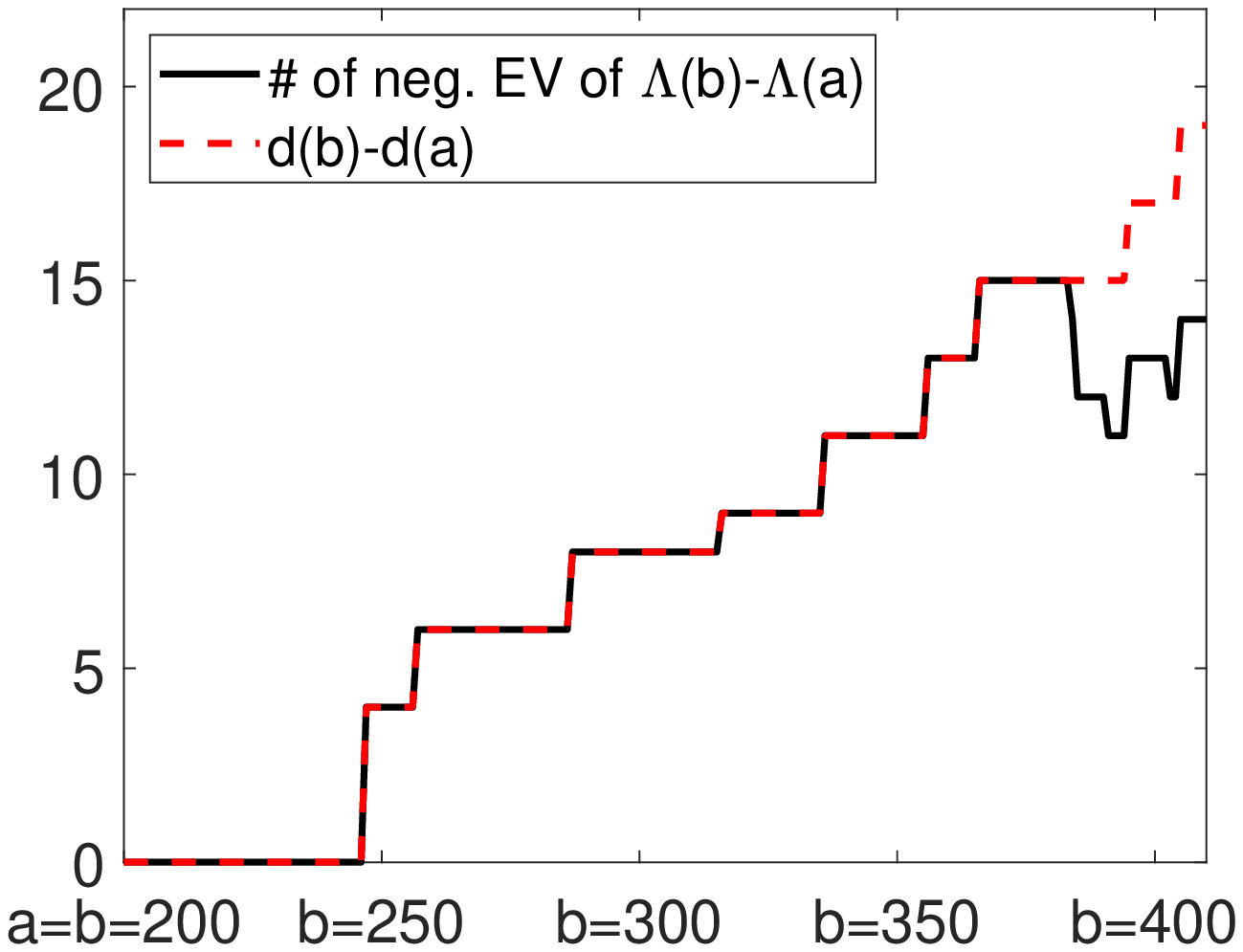}}
\end{tabular}
\end{center}
\caption{Comparison of the numerically calculated number of negative eigenvalues of $\Lambda(b)-\Lambda(a)$ 
 with the theoretical bound $d(b)-d(a)$.}
\label{fig:Comparison_Plot}
\end{figure}

To further investigate this additional effect, figure \ref{fig:EV_plot} shows the values of the eigenvalues of $\Lambda(b)-\Lambda(a)$ as a function of $b$ for fixed $a=-10$. More precisely, for each integer $b=-10,-9,\ldots,200$ (excluding the resonance $b=0$), the black dots are plotted at the position $(b,\lambda_j(a,b))$ where 
$\lambda_j(a,b)$, $j=1,\ldots,1000$, are the numerically calculated eigenvalues of $\Lambda(b)-\Lambda(a)$.     
The red dashed lines show the positions of the Neumann eigenvalues. Whenever $b$ crosses an eigenvalue, new negative eigenvalues of $\Lambda(b)-\Lambda(a)$ appear.
But at the same time the values of the eigenvalues increase with $b-a$, and it seems that negative eigenvalues can become positive again which would explain the drops in the number of negative eigenvalues observed in figure \ref{fig:Comparison_Plot}. 

\begin{figure}
\begin{center}
\begin{tabular}{c}
\mbox{\includegraphics[width=12.25cm,trim=20 0 10 0, clip]{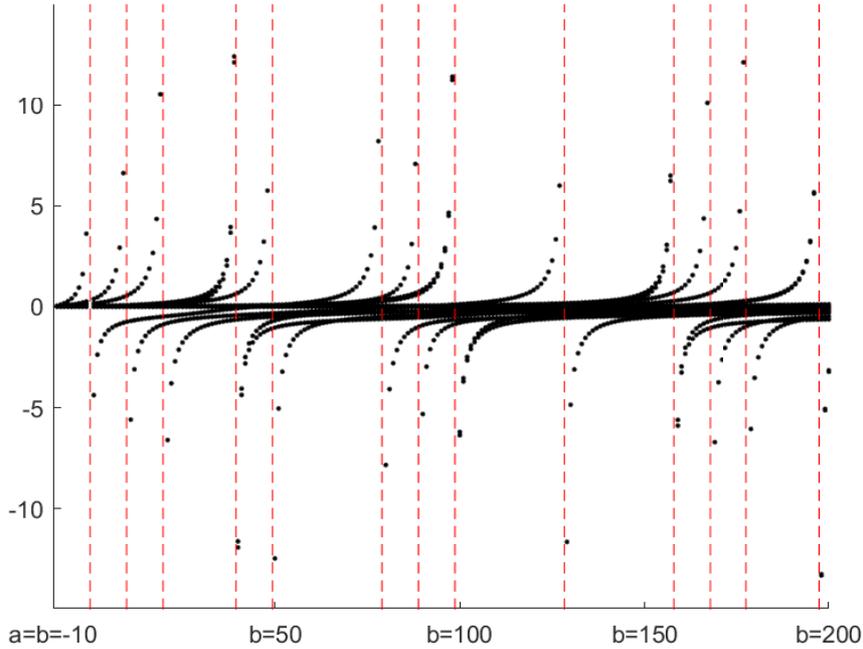}}
\end{tabular}
\end{center}
\caption{Plot of the numerically calculated eigenvalues of $\Lambda(b)-\Lambda(a)$ (black dots) 
as a function of $b$ for fixed $a:=-10$.}
\label{fig:EV_plot}
\end{figure}

Let us stress however that this numerical experiment is only an indication of what might happen to stipulate further research. We do not have a rigorous proof that the observed drop in the number of negative eigenvalues really exists. $\Lambda(b)-\Lambda(a)$ is a compact operator with an infinite number of eigenvalues accumulating at zero, and we cannot rigorously rule out the possibility that there exist more negative eigenvalues (up to the theoretically proven bound $d(b)-d(a)$) that we did not find due to their absolute values being below the numerical precision level.

\bibliographystyle{alpha}
\bibliography{literaturliste}

\newcommand{\etalchar}[1]{$^{#1}$}
\begin{thebibliography}{VMC{\etalchar{+}}17}

\bibitem[AH13]{arnold2013unique}
Lilian Arnold and Bastian Harrach.
\newblock Unique shape detection in transient eddy current problems.
\newblock {\em Inverse Problems}, 29(9):095004, 2013.

\bibitem[BHHM17]{barth2017detecting}
Andrea Barth, Bastian Harrach, Nuutti Hyv{\"o}nen, and Lauri Mustonen.
\newblock Detecting stochastic inclusions in electrical impedance tomography.
\newblock {\em Inverse Problems}, 33(11):115012, 2017.

\bibitem[BHKS18]{brander2018monotonicity}
Tommi Brander, Bastian Harrach, Manas Kar, and Mikko Salo.
\newblock Monotonicity and enclosure methods for the {$p$}-{L}aplace equation.
\newblock {\em SIAM J. Appl. Math.}, 78(2):742--758, 2018.

\bibitem[BS12]{berezin2012schrodinger}
Feliks~Aleksandrovich Berezin and M~Shubin.
\newblock {\em The {S}chr{\"o}dinger Equation}, volume~66.
\newblock Springer Science \& Business Media, 2012.

\bibitem[DS67]{dunford1967linear}
N.~Dunford and J.~T. Schwartz.
\newblock {\em Linear operators {I}--{II}}.
\newblock Interscience Publishers, 3rd printing, 1967.

\bibitem[Gar17]{garde2017comparison}
Henrik Garde.
\newblock Comparison of linear and non-linear monotononicity-based shape
  reconstruction using exact matrix characterizations.
\newblock {\em Inverse Problems in Science and Engineering}, 2017.

\bibitem[Geb08]{gebauer2008localized}
Bastian Gebauer.
\newblock Localized potentials in electrical impedance tomography.
\newblock {\em Inverse Probl. Imaging}, 2(2):251--269, 2008.

\bibitem[GH18]{griesmaier2018monotonicity}
Roland Griesmaier and Bastian Harrach.
\newblock Monotonicity in inverse medium scattering on unbounded domains.
\newblock {\em SIAM J. Appl. Math}, 78(5):2533--2557, 2018.

\bibitem[GM08]{gesztesy2008robin}
Fritz Gesztesy and Marius Mitrea.
\newblock Generalized {R}obin boundary conditions, {R}obin-to-{D}irichlet maps,
  and {K}rein-type resolvent formulas for {S}chr\"{o}dinger operators on
  bounded {L}ipschitz domains.
\newblock In {\em Perspectives in partial differential equations, harmonic
  analysis and applications}, volume~79 of {\em Proc. Sympos. Pure Math.},
  pages 105--173. Amer. Math. Soc., Providence, RI, 2008.

\bibitem[GN13]{grebenkov2013geometrical}
Denis~S Grebenkov and B-T Nguyen.
\newblock Geometrical structure of laplacian eigenfunctions.
\newblock {\em SIAM Review}, 55(4):601--667, 2013.

\bibitem[GS17]{garde2017convergence}
Henrik Garde and Stratos Staboulis.
\newblock Convergence and regularization for monotonicity-based shape
  reconstruction in electrical impedance tomography.
\newblock {\em Numerische Mathematik}, 135(4):1221--1251, 2017.

\bibitem[GS19]{garde2019regularized}
Henrik Garde and Stratos Staboulis.
\newblock The regularized monotonicity method: Detecting irregular indefinite
  inclusions.
\newblock {\em Inverse Probl.\ Imaging}, 13(1):93--116, 2019.

\bibitem[Har09]{harrach2009uniqueness}
Bastian Harrach.
\newblock On uniqueness in diffuse optical tomography.
\newblock {\em Inverse Problems}, 25:055010 (14pp), 2009.

\bibitem[Har12]{harrach2012simultaneous}
Bastian Harrach.
\newblock Simultaneous determination of the diffusion and absorption
  coefficient from boundary data.
\newblock {\em Inverse Probl. Imaging}, 6(4):663--679, 2012.

\bibitem[Har19]{harrach2019uniqueness}
Bastian Harrach.
\newblock Uniqueness and {L}ipschitz stability in electrical impedance
  tomography with finitely many electrodes.
\newblock {\em Inverse Problems}, 35(2):024005, 2019.

\bibitem[HL18]{harrach2018fractional}
Bastian Harrach and Yi-Hsuan Lin.
\newblock Monotonicity-based inversion of the fractional {S}chr\"odinger
  equation {I}. {P}ositive potentials.
\newblock {\em arXiv preprint arXiv:1711.05641}, 2018.

\bibitem[HL19]{harrach2019fractional}
Bastian Harrach and Yi-Hsuan Lin.
\newblock Monotonicity-based inversion of the fractional {S}chr\"odinger
  equation {II}. {G}eneral potentials and stability.
\newblock {\em arXiv preprint arXiv:1903.08771}, 2019.

\bibitem[HLL18]{harrach2018localizing}
Bastian Harrach, Yi-Hsuan Lin, and Hongyu Liu.
\newblock On localizing and concentrating electromagnetic fields.
\newblock {\em SIAM J. Appl. Math}, 78(5):2558--2574, 2018.

\bibitem[HLU15]{harrach2015combining}
Bastian Harrach, Eunjung Lee, and Marcel Ullrich.
\newblock Combining frequency-difference and ultrasound modulated electrical
  impedance tomography.
\newblock {\em Inverse Problems}, 31(9):095003, 2015.

\bibitem[HM16]{harrach2016enhancing}
Bastian Harrach and Mach~Nguyet Minh.
\newblock Enhancing residual-based techniques with shape reconstruction
  features in electrical impedance tomography.
\newblock {\em Inverse Problems}, 32(12):125002, 2016.

\bibitem[HM18]{harrach2018monotonicity}
Bastian Harrach and Mach~Nguyet Minh.
\newblock Monotonicity-based regularization for phantom experiment data in
  electrical impedance tomography.
\newblock In {\em New Trends in Parameter Identification for Mathematical
  Models}, pages 107--120. Springer, 2018.

\bibitem[HM19]{harrach2019global}
Bastian Harrach and Houcine Meftahi.
\newblock Global uniqueness and {L}ipschitz-stability for the inverse {R}obin
  transmission problem.
\newblock {\em SIAM J. Appl. Math.}, 79(2):525--550, 2019.

\bibitem[HPS]{harrach2018helmholtz}
Bastian Harrach, Valter Pohjola, and Mikko Salo.
\newblock Monotonicity and local uniqueness for the helmholtz equation.
\newblock {\em to appear in Anal. PDE}.

\bibitem[HS10]{harrach2010exact}
Bastian Harrach and Jin~Keun Seo.
\newblock Exact shape-reconstruction by one-step linearization in electrical
  impedance tomography.
\newblock {\em SIAM Journal on Mathematical Analysis}, 42(4):1505--1518, 2010.

\bibitem[HU13]{harrach2013monotonicity}
Bastian Harrach and Marcel Ullrich.
\newblock Monotonicity-based shape reconstruction in electrical impedance
  tomography.
\newblock {\em SIAM Journal on Mathematical Analysis}, 45(6):3382--3403, 2013.

\bibitem[HU15]{harrach2015resolution}
Bastian Harrach and Marcel Ullrich.
\newblock Resolution guarantees in electrical impedance tomography.
\newblock {\em IEEE Trans. Med. Imaging}, 34:1513--1521, 2015.

\bibitem[HU17]{harrach2017local}
Bastian Harrach and Marcel Ullrich.
\newblock Local uniqueness for an inverse boundary value problem with partial
  data.
\newblock {\em Proceedings of the American Mathematical Society},
  145(3):1087--1095, 2017.

\bibitem[HW08]{hardy2008introduction}
G.~H. Hardy and E.~M. Wright.
\newblock {\em An introduction to the theory of numbers}.
\newblock Oxford University Press, Oxford, sixth edition, 2008.

\bibitem[Kat95]{kato1995perturbation}
Tosio Kato.
\newblock {\em Perturbation theory for linear operators}.
\newblock Classics in Mathematics. Springer-Verlag, Berlin, 1995.

\bibitem[MVVT16]{maffucci2016novel}
Antonio Maffucci, Antonio Vento, Salvatore Ventre, and Antonello Tamburrino.
\newblock A novel technique for evaluating the effective permittivity of
  inhomogeneous interconnects based on the monotonicity property.
\newblock {\em IEEE Transactions on Components, Packaging and Manufacturing
  Technology}, 6(9):1417--1427, 2016.

\bibitem[Rem91]{remmert1991theory}
Reinhold Remmert.
\newblock {\em Theory of complex functions}, volume 122 of {\em Graduate Texts
  in Mathematics}.
\newblock Springer-Verlag, New York, 1991.
\newblock Translated from the second German edition by Robert B. Burckel,
  Readings in Mathematics.

\bibitem[RR04]{renardy2004introduction}
Michael Renardy and Robert~C. Rogers.
\newblock {\em An introduction to partial differential equations}, volume~13 of
  {\em Texts in Applied Mathematics}.
\newblock Springer-Verlag, New York, second edition, 2004.

\bibitem[RS72]{RSI}
Michael Reed and Barry Simon.
\newblock {\em Methods of modern mathematical physics, Volume I: Functional
  analysis}.
\newblock Academic Press, San Diego, 1972.

\bibitem[RS78]{reed1978methods}
Michael Reed and Barry Simon.
\newblock {\em Methods of modern mathematical physics. {IV}. {A}nalysis of
  operators}.
\newblock Academic Press, New York-London, 1978.

\bibitem[SKJ{\etalchar{+}}18]{seo2018learning}
Jin~Keun Seo, Kang~Cheol Kim, Ariungerel Jargal, Kyounghun Lee, and Bastian
  Harrach.
\newblock A learning-based method for solving ill-posed nonlinear inverse
  problems: a simulation study of lung eit.
\newblock {\em arXiv preprint arXiv:1810.10112}, 2018.

\bibitem[SUG{\etalchar{+}}17]{su2017monotonicity}
Zhiyi Su, Lalita Udpa, Gaspare Giovinco, Salvatore Ventre, and Antonello
  Tamburrino.
\newblock Monotonicity principle in pulsed eddy current testing and its
  application to defect sizing.
\newblock In {\em Applied Computational Electromagnetics Society
  Symposium-Italy (ACES), 2017 International}, pages 1--2. IEEE, 2017.

\bibitem[TR02]{tamburrino2002new}
Antonello Tamburrino and Guglielmo Rubinacci.
\newblock A new non-iterative inversion method for electrical resistance
  tomography.
\newblock {\em Inverse Problems}, 18(6):1809, 2002.

\bibitem[TSV{\etalchar{+}}16]{tamburrino2016monotonicity}
Antonello Tamburrino, Zhiyi Sua, Salvatore Ventre, Lalita Udpa, and Satish~S
  Udpa.
\newblock Monotonicity based imang method in time domain eddy current testing.
\newblock {\em Electromagnetic Nondestructive Evaluation (XIX)}, 41:1, 2016.

\bibitem[VMC{\etalchar{+}}17]{ventre2017design}
Salvatore Ventre, Antonio Maffucci, Fran{\c{c}}ois Caire, Nechtan Le~Lostec,
  Antea Perrotta, Guglielmo Rubinacci, Bernard Sartre, Antonio Vento, and
  Antonello Tamburrino.
\newblock Design of a real-time eddy current tomography system.
\newblock {\em IEEE Transactions on Magnetics}, 53(3):1--8, 2017.

\bibitem[ZHS18]{zhou2018monotonicity}
Liangdong Zhou, Bastian Harrach, and Jin~Keun Seo.
\newblock Monotonicity-based electrical impedance tomography for lung imaging.
\newblock {\em Inverse Problems}, 34(4):045005, 2018.

\end{thebibliography}

\end{document}